\documentclass[11pt]{article}

\usepackage[left=1in,top=1in,right=1in,bottom=1in,letterpaper]{geometry}

\usepackage{amsmath,amsfonts}
\usepackage{array}
\usepackage[caption=false,font=normalsize,labelfont=sf,textfont=sf]{subfig}
\usepackage{textcomp}
\usepackage{stfloats}
\usepackage{url}
\usepackage{verbatim}
\usepackage{graphicx}
\usepackage{balance}

\usepackage[utf8]{inputenc}
\usepackage{amsfonts}
\usepackage{amssymb,xcolor}
\usepackage{amsmath,mathtools,amsthm}

\usepackage[utf8]{inputenc} 
\usepackage[T1]{fontenc}    
\usepackage{url}            
\usepackage{booktabs}       
\usepackage{amsfonts}       
\usepackage{nicefrac}       
\usepackage{microtype}      
\usepackage{amssymb}
\usepackage{amsmath,mathtools}

\usepackage{algorithm}
\usepackage[noend]{algpseudocode}

\usepackage[round]{natbib}
\bibpunct[, ]{(}{)}{,}{a}{}{,}%

\usepackage{float}
\usepackage{color}
\usepackage{amsfonts}
\usepackage{wrapfig}
\usepackage{mwe}

\usepackage{times}
\usepackage{epsfig}
\usepackage{graphicx}
\usepackage{amsmath}
\usepackage{amssymb}

\newtheorem{assumption}{Assumption}[section]
\newtheorem{definition}{Definition}[section]
\newtheorem{lemma}{Lemma}[section]
\newtheorem{theorem}{Theorem}[section]
\newtheorem{remark}{Remark}[section]
\newtheorem{proposition}{Proposition}[section]

\usepackage{mathtools}
\usepackage{bbm,nicefrac}
\usepackage{multicol, multirow}
\usepackage{mathtools}
\usepackage{enumitem, xcolor}

\title{Stochastic Zeroth-order Functional Constrained Optimization: Oracle Complexity and Applications}

\author{Anthony Nguyen\thanks{Department of Mathematics, University of California, Davis. \texttt{antngu@ucdavis.edu}.  } 
\and Krishnakumar Balasubramanian\thanks{Department of Statistics, University of California, Davis \texttt{kbala@ucdavis.edu}. } 
}

\begin{document}
\maketitle

\begin{abstract}
Functionally constrained stochastic optimization problems, where neither the objective function nor the constraint functions are analytically available, arise frequently in machine learning applications. In this work, assuming we only have access to the noisy evaluations of the objective and constraint functions, we propose and analyze stochastic zeroth-order algorithms for solving the above class of stochastic optimization problem. When the domain of the functions is $\mathbb{R}^n$, assuming there are $m$ constraint functions, we establish oracle complexities of order $\mathcal{O}((m+1)n/\epsilon^2)$ and $\mathcal{O}((m+1)n/\epsilon^3)$ respectively in the convex and nonconvex setting, where $\epsilon$ represents the accuracy of the solutions required in appropriately defined metrics. The established oracle complexities are, to our knowledge, the first such results in the literature for functionally constrained stochastic zeroth-order optimization problems. We demonstrate the applicability of our algorithms by illustrating its superior performance on the problem of hyperparameter tuning for sampling algorithms and neural network training.
\end{abstract}

\section{Introduction}\label{sec:intro}
We develop and analyze stochastic zeroth-order algorithms for solving the following non-linear optimization problem with functional constraints: 
\begin{equation}
\begin{aligned} 
    ~~ \min_{x \in X}~~f_0(x)\quad \text{such that}\quad f_i(x) \leqslant 0, \quad \color{black}{i \in \{0, 1, \dots, m\}},\label{eqn:1.1}
\end{aligned}
\end{equation}
where, for $i \in \{0, 1, \dots, m\}$, \textcolor{black}{$f_i: \mathbb{R}^n \to \mathbb{R}$ are continuous functions which are not necessarily convex defined as $f_i(x) = \mathbb{E}_{\xi_i}[F_i(x, \xi_i)]$ with $\xi_i$ denoting the noise vector associated with function $f_i$}, and  $X \subseteq \mathbb{R}^n$ is a convex compact set \textcolor{black}{that represents \emph{known} constraints (i.e., constraints that are analytically available)}. In the stochastic zeroth-order setting, we \emph{neither observe the objective function $f_0$ nor the constraint functions $f_i$ analytically}. We only have access to noisy function evaluations of them. The study of stochastic zeroth-order optimization algorithms for unconstrained optimization problems goes back to the early works of~\cite{kiefer1952stochastic, blum1954multidimensional, hooke1961direct, spendley1962sequential, powell1964efficient, nelder1965simplex,  nemirovskij1983problem, spall1987stochastic}. Such zeroth-order algorithms have 
proved to be extremely useful for hyperparameter tuning~\citep{snoek2012practical, hernandez2015predictive, gelbart2014bayesian, ruan2019learning, golovin2017google}, reinforcement learning~\citep{mania2018simple,salimans2017evolution, gao2020robotic, choromanski2020provably} and robotics~\citep{jaquier2020bayesian, jaquier2020high}. However, the study of zeroth-order algorithms and their oracle complexities for constrained problem as in~\eqref{eqn:1.1} is limited, despite the fact that several real-world machine learning problems fall under the setting of~\eqref{eqn:1.1}. We now describe two such applications that serve as  our main motivation for developing stochastic zeroth-order optimization algorithms for solving~\eqref{eqn:1.1}, and analyzing their oracle complexity.

\subsection{Motivating application I:} Hamiltonian Monte Carlo (HMC) algorithm, proposed by~\cite{duane1987hybrid} and popularized in the statistical machine learning community by~\cite{neal2011mcmc}{\color{black}{,}} is a gradient-based sampling algorithm that works by discretizing the continuous time degenerate Langevin diffusion~\citep{leimkuhler2015molecular}. It has been used successfully as a state-of-the art sampler or a numerical integrator in the Bayesian statistical machine learning community by~\cite{hoffman2014no, wang2013adaptive, girolami2011riemann, chen2014stochastic,carpenter2017stan}.  However, in order to obtain successful performance in practice using HMC, several hyperparameters need to be tuned optimally. Typically, the functional relationship between the hyperparameters that need to be tuned and the performance measure used is not available in an analytical form. We can only evaluate the performance of the sampler for various settings of the hyperparameter. Furthermore, in practice several constraints, for example, constraints on running times and constraints that enforce the generated samples to pass certain standard diagnostic tests~\citep{geweke1991evaluating, gelman1992single}, are enforced in the hyperparameter tuning process. The functional relationship between such constraints and the hyperparameters is also not available analytically. This makes the problem of optimally setting the hyperparameters for HMC a constrained zeroth-order optimization problem. As a preview, in Section~\ref{sec:hmctuning}, we show that our approach provides significant improvements over existing methods of~\cite{mahendran2012adaptive, gelbart2014bayesian, hernandez2015predictive}{\color{black}{,}} which are based on Bayesian optimization techniques for tuning HMC{\color{black}{,}} when we measure the performance \textcolor{black}{adopting} the widely used \emph{effective sample size} metric~\citep{kass1998markov}.

\subsection{Motivating application II:} Deep learning has achieved state-of-the-art performance in the recent years for various prediction tasks~\citep{goodfellow2016deep}. Among the various factors involved behind the success of deep learning, hyperparameter tuning is one of primary factors~\citep{snoek2012practical, bergstra2012random, li2017hyperband, hazan2018hyperparameter, elsken2019neural}. However, most of the existing methods for tuning the hyperparameters do not enforce any constraints on the prediction time required on the validation set or memory constraints on the training algorithm. Such constraints are typically required to make deep learning methods widely applicable to problem arising in
several consumer applications based on tiny devices~\citep{perera2015emerging, latre2011survey, yang2008distributed}. As in the above motivating application, the  functional relationship between such constraints and the hyperparameters is not available analytically. As a preview, in Section~\ref{sec:dnntuning}, we show that our approach provides significant improvements over the existing works of~\cite{gelbart2014bayesian, hernandez2015predictive, ariafar2019admmbo} that developed hyperparameter tuning techniques which explictly take into account time/memory constraints. 

\subsection{Related works} 

\textcolor{black}{In the operations research and statistics communities, zeroth-order optimization techniques are well-studied under the name of derivative-free optimization. Interested readers are referred to~\cite{conn2009introduction} and~\cite{audet2017derivative}.} In the machine learning community, Bayesian optimization techniques have been developed for optimizing functions with only noisy function evaluations. We refer the reader to~\cite{mockus1994application, kolda2003optimization, spall2005introduction, conn2009introduction, mockus2012bayesian, brent2013algorithms, shahriari2015taking, audet2017derivative, larson2019derivative, frazier2018tutorial, archetti2019bayesian, liu2020primer} for more details. In what follows, we focus on relevant literature from zeroth-order optimization and Bayesian optimization literature for \emph{known constrained optimization} problems \textcolor{black}{(i.e., problems with constraints that are analytically available)}. When the constraint set is analytically available and only the objective function is not,~\cite{lewis2002globally} and~\cite{bueno2013inexact} considered an augmented Lagrangian approach and an inexact restoration method respectively, and provided convergence analysis. Furthermore,~\cite{kolda2003optimization, amaioua2018efficient, audet2015linear} extended the popular mesh adaptive direct search to this setting. Projection-free methods based on Frank-Wolfe methods have been considered in~\cite{balasubramanian2018zero, sahu2019towards} for the case when the constraint set is a convex subset of $\mathbb{R}^n$. Furthermore,~\cite{li2020zeroth} considered the case when the constraint set is a Riemannian submanifold embedded in $\mathbb{R}^n$ (and the function is defined only over the manifold). None of the above works are directly applicable to the case of unknown constraints that we consider in this work.

We now discuss some existing methods for solving (variants of) problem~\eqref{eqn:1.1} in the zeroth-order setting.  For solving~\eqref{eqn:1.1} in the deterministic setting (i.e., we could obtain exact evaluations of the objective and the constraint functions at a given point), \emph{filter methods} which reduce the objective function while trying to reduce constraint violations were proposed and analyzed in~\cite{audet2004pattern, echebest2017inexact, pourmohamad2020statistical}. Barrier methods in the zeroth-order setting were considered in~\cite{audet2006mesh, audet2009progressive, liuzzi2009derivative, gratton2014merit, fasano2014linesearch, liuzzi2010sequential, dzahini2020constrained}, with some works also developing line search approaches for setting the tuning parameters. Model based approaches were considered in the works of~\cite{muller2017gosac, troltzsch2016sequential, augustin2014nowpac, gramacy2016modeling, conn2013use}. Furthermore,~\cite{bHurmen2006grid, audet2018mesh} developed extensions of Nelder--Mead algorithm to the constrained setting. 

\textcolor{black}{S}everal works in the statistical machine learning community also considered Bayesian optimization methods in the constrained setting, in both the noiseless and noisy setting. We refer the reader, for example, to ~\cite{gardner2014bayesian, gelbart2016general, ariafar2019admmbo,  balandat2020botorch, bachoc2020gaussian, greenhill2020bayesian, eriksson2020scalable, letham2019constrained, hernandez2015predictive, lam2017lookahead, picheny2016bayesian, acerbi2017practical}. On one hand, the above works demonstrate the interest in the optimization and machine learning communities for developing algorithms for constrained zeroth-order optimization problems. On the other hand, most of the above works are not designed to handle \emph{stochastic} zeroth-order constrained optimization that we consider. Furthermore, a majority of the above works are methodological, and the few works that develop convergence analysis do so only in the asymptotic setting. \textcolor{black}{A recent work by~\cite{usmanova2019safe}} considered the case when the constraints are linear functions (but unknown), and provided a Frank-Wolfe based algorithm with \emph{estimated} constraints. However, the proposed approach is limited to only linear constraints, and the oracle complexities established are highly sub-optimal. To the best of our knowledge, there is no rigorous non-asymptotic analysis of the oracle complexity of stochastic zeroth-optimization when the constraints and the objective values are available only via \emph{noisy} function evaluations.

\subsection{Methodology and Main Contributions:} Our methodology is based on a novel constraint extrapolation technique developed for the zeroth-order setting, extending the work of~\cite{boob2019proximal} in the first-order setting, and the Gaussian smoothing based zeroth-order stochastic gradient estimators. Specifically, we propose the~\texttt{SZO-ConEX} method in Algorithm~\ref{Algorithm1} for solving problems of the form in~\eqref{eqn:1.1}. We theoretically characterize how to set the \emph{tuning parameters} of the algorithm so as to mitigate the issues caused by the \emph{bias} in the stochastic zeroth-order gradient estimates and obtain the oracle complexity of our algorithm. More specifically, we make the following main contributions: 
\begin{itemize}[leftmargin=0.2in]
\item When the functions $f_i$, $i=0,\ldots, m$, are convex, in Theorem~\ref{thm11}, we show that the number of calls to the stochastic zeroth-order oracle to achieve an appropriately defined $\epsilon$-optimal solution of~\eqref{eqn:1.1} (see Definition~\ref{def:epsconvex}) is of order $\mathcal{O}((m+1)n/\epsilon^2)$.
\item When the functions are nonconvex, in Proposition~\ref{prop:nonconvex}, we show that the number of calls to the stochastic zeroth-order oracle to \textcolor{black}{achieve} an appropriately defined $\epsilon$-optimal KKT solution of~\eqref{eqn:1.1} (see Definition~\ref{def:epsnonconvex}) is of order $\mathcal{O}((m+1)n/\epsilon^3)$. 
\end{itemize}
To our knowledge, these are the first non-asymptotic oracle complexity \textcolor{black}{results} for stochastic zeroth-order optimization with stochastic zeroth-order functional constraints. We illustrate the practical applicability of the developed methodology by testing its performance on hyperparameter tuning for HMC sampling algorithm (Section~\ref{sec:hmctuning}) and 3-layer neural network (Section~\ref{sec:dnntuning}). 

\section{Preliminaries and Methodology}\label{sec:prelim}
\textbf{Notations:} Let $\mathbf{0}$ denote the vector of elements $0$ and $[m] \coloneqq \{1, \dots, m\}$. Let $f(x) \coloneqq [f_1(x), \dots, f_m(x)]^T$; then, the constraints in \eqref{eqn:1.1} \textcolor{black}{can} be expressed as $f(x) \leqslant \mathbf{0}$.  We use $\xi \coloneqq [\xi_1, \cdots, \xi_m]$ to denote the random vectors in the constraints. Furthermore, $\| \cdot \|$ denotes a general norm and $\| \cdot \|_*$ denotes its dual norm defined as $\| z \|_* \coloneqq \sup\{ z^T x : \| x \| \leq 1\}$. 
Furthermore, \quad $[x]_+ \coloneqq \max\{x, 0\}$ for any $x \in \mathbb{R}$. For any vector $x \in \mathbb{R}^k$, we define $[x]_+$ as element-wise application of $[\cdot ]_+$. 

We now describe the precise assumption made on the \emph{stochastic zeroth-order oracle} in this work.
 
\begin{assumption} \label{Assumption1}
Let $\| \cdot \|$ be a norm on $\mathbb{R}^n$. For $i \in \{0, \dots, m\}$ and for any $x \in \mathbb{R}^n$, the zeroth-order oracle outputs an estimator $F_i(x, \xi_i)$ of $f_i(x)$ such that $\mathbb{E}[F_i(x, \xi_i)] = f_i(x)$, $\mathbb{E}[F_i(x, \xi_i)^2] \leq \sigma^2_{f_i}$, $\mathbb{E}[\nabla F_i(x, \xi_i)] = \nabla f_i(x)$, $\mathbb{E}[\| \nabla F_i(x, \xi_i) - \nabla f_i(x)\|_*^2] \leqslant \sigma_i^2$, , where $\| \cdot \|_*$ denotes the dual norm.
\end{assumption}
The assumption above assumes that we have access to a stochastic zeroth-order oracle which provides unbiased function evaluations with bounded variance. It is worth noting that in the above assumption, we do not necessarily assume the noise $\xi_i$ is additive. Furthermore, we allow for different noise models for the objective function and the $m$ constraint functions, which is a significantly general model compared to several existing works \textcolor{black}{such as}~\cite{digabel2015taxonomy}. Our gradient estimator is then constructed by leveraging the Gaussian smoothing technique \textcolor{black}{proposed in}~\cite{nemirovskij1983problem, nesterov2017random}. For $\nu_i \in (0, \infty)$ we introduce the smoothed function $f_{i, \nu_{i}}(x) = \mathbb{E}_{u_i}[f_i(x + \nu_i u_i)]$ where $u_i \sim N(0, I_n)$ and independent across $i$. We can estimate the gradient of this smoothed function using function evaluations of $f_i$. Specifically, we define the stochastic zeroth-order gradient of $f_{i, \nu_i}(x)$ as 
\begin{align}\label{eq:gradest}
    G_{i, \nu_i}(x, \xi_i, u_i) = \frac{F_i(x + \nu_i u_i, \xi_i) - F_i(x, \xi_i)}{\nu_i}u_i,
\end{align}
which is an unbiased estimator of $\nabla f_{i, \nu_i}(x)$, i.e., we have $\mathbb{E}_{u, \xi_i}[G_{i, \nu_i}(x, \xi_i, u)] = \nabla f_{i, \nu_i}(x)$. However, it is well-known that $ G_{i, \nu_i}(x, \xi_i, u_i)$ is a biased estimator of $\nabla f_i(x)$. An interpretation of the gradient estimator in~\eqref{eq:gradest} as a consequence of Gaussian Stein's identity, popular in the statistics literature~\textcolor{black}{\citep{stein1972bound}}, was provided in~\cite{balasubramanian2018zeroth}. 

The gradient estimator in~\eqref{eq:gradest} is referred to as the two-point estimator in the literature. The reason is that, for a given random vector $\xi_i$, it is assumed that the stochastic function in~\eqref{eq:gradest} could be evaluated at two points, $F_{\color{black}{i}}(x + \nu_i u_i,\xi_i)$ and $F_{\color{black}{i}}(x,\xi_i)$. Such an assumption is satisfied in several statistics, \textcolor{black}{machine learning, simulation based optimization, and sampling problems}; see for example~\cite{spall2005introduction, mokkadem2007companion, dippon2003accelerated, agarwal2010optimal, duchi2015optimal, ghadimi2013stochastic, nesterov2017random}. Yet another estimator in the literature is the one-point estimator, which assumes that for each $\xi_i$, we observe only one noisy function evaluation $F_{\color{black}{i}}(x + \nu_i u_i,\xi_i)$. It is well-known that the one-point setting is more challenging than the two-point setting~\textcolor{black}{\citep{shamir2013complexity}}. From a theoretical point of view, the use of two-point evaluation based gradient estimator is primarily motivated by the sub-optimality (in terms of oracle complexity) of one-point feedback based stochastic zeroth-order optimization methods either in terms of the approximation accuracy or dimension dependency. For the rest of this work, we focus on the two-point setting and leave the question of obtaining results in the one-point setting as future work. We now describe our assumptions on the objective and constraint functions. 
\begin{assumption} \label{Assumption3}
Function $F_i$ has Lipschitz continuous gradient with constant $L_i$, almost surely for any $\xi_i$, i.e., $\| \nabla F_i(y, \xi_i) - \nabla F_i(x, \xi_i)\|_* \leqslant L_i \| y - x \|$, which consequently implies that $|F_i(y, \xi_i) - F_i(x, \xi_i) - \langle \nabla F_i(x, \xi_i), y - x\rangle | \leqslant \frac{L_i}{2}\| y - x\|^2$ for $i \in \{0, 1, \dots, m\}$.
\end{assumption}
\begin{assumption} \label{Assumption2}
Function $F_i$ is Lipschitz continuous with constant $M_{i}$, almost surely for any $\xi_i$, i.e., $| F_i(y, \xi_i) - F_i(x, \xi_i) | \leq M_{i}\|y - x\|$, for $i \in \{0, 1, \dots, m\}$. 
\end{assumption}
The above smoothness assumptions are standard in the literature on stochastic zeroth-order optimization and are made in several works~\cite{nesterov2017random,ghadimi2013stochastic, balasubramanian2018zeroth} for obtaining oracle complexity results. It is easy to see that Assumption \ref{Assumption3} implies that for $i \in \{0, \dots, m\}$, $f_i$ has Lipschitz continuous gradient with constant $L_i$ since $\| \nabla f_i(y) - \nabla f_i(x) \|_* \leqslant \mathbb{E}[\| \nabla F(y, \xi) - \nabla F(x, \xi)\|_*] \leqslant L_i\| y - x\|$, due to Jensen's inequality for the dual norm. By similar reasoning and Assumption \ref{Assumption2}, we also see that $f_i$ is Lipschitz continuous with constant $M_{i}$. Due to Assumptions \ref{Assumption3} and~\ref{Assumption2}, we also have $\| f(x_1) - f(x_2) \|_2  \leqslant M_f \|x_1 - x_2\|$, $\| \nabla f(x_2)^T(x_1 - x_2)\|_2 \leqslant M_f\|x_1 - x_2\|$ and $\| f(x_1) - f(x_2) - \nabla f(x_2)^T(x_1 - x_2)\|_2 \leqslant \frac{L_f}{2}\|x_1 - x_2\|^2$, for all $x_1, x_2 \in \mathbb{R}^n$, where $\nabla f(\cdot) := [\nabla f_1(\cdot), \dots, \nabla f_m(\cdot)] \in \mathbb{R}^{n \times m}$ and constants $M_f$ and $L_f$ are defined as 
\begin{align}
M_f \coloneqq \sqrt{\textstyle\sum_{i=1}^mM_{i}^2}~~~\text{and}~L_f \coloneqq \sqrt{\textstyle\sum_{i=1}^mL_i^2}. \label{eq:1.10}
\end{align}

We now state the definition of the \textbf{prox}-function and the \textbf{prox}-operator. The class of algorithms based on \textbf{prox}-operators are called proximal algorithms. Such algorithms have turned out to be particularly useful for efficiently solving various machine learning problems in the recent past. We refer the interested reader to~\cite{parikh2014proximal, beck2017first} for more details.
\begin{definition}\label{def:www}
Let $\omega : X \to \mathbb{R}$ be continuously differentiable, $L_{\omega}$-Lipschitz gradient smooth, and $1$-strongly convex with respect to $\| \cdot \|$ function. We define the \textbf{prox}-function associated with $\omega(\cdot)$, $\forall x, y \in \mathbb{R}^n$, as $W(y, x)\coloneqq \omega(y) - \omega(x) - \langle \nabla \omega(x), y - x\rangle$. Based on the smoothness and strong convexity of $\omega(x)$, we have $W(y, x) \leqslant \frac{L_{\omega}}{2}\| x - y \|^2 \leqslant L_{\omega} W(x, y)$, $\forall x, y \in \mathbb{R}^n$. For any $v \in \mathbb{R}^n$, we define the following \textbf{prox}-operator as $\textbf{prox}(v, \tilde{x}, \eta) := \arg\min_{x \in X}\{\langle v, x\rangle + \eta W(x, \tilde{x})\}$.
\end{definition}
The function $W$ is also called as Bregman divergence in the literature. A canonical example of $W$ is that of the Euclidean distance function $\| x - y \|^2$ which is useful when $X=\mathbb{R}^n$. We will see in Section~\ref{sec:methods}~that our algorithm is based on the above \textbf{prox}-operator. Finally, we have the following results which will prove to be useful for subsequent calculations. Let $u\coloneqq [u_1, \cdots, u_m]$ and $D_X\coloneqq \sup_{x,y \in X} \sqrt{W(x,y)}$ be the diameter of the set $X${\color{black}{.}}

\begin{lemma}\label{lemma:zogradbounds}
Let $\nu \coloneqq [\nu_1,\cdots, \nu_m]$, $F_{\nu}(x, \xi, u)\coloneqq[F_1(x + \nu_1u_1, \xi_1), \ldots, F_m(x+ \nu_m u_m, \xi_m)]^T$ and $f_{\nu}(x)\coloneqq [f_{1, \nu_1}(x), \ldots, f_{m, \nu_m}(x)]^T$. Under assumption \ref{Assumption2}, we have $\mathbb{E}_{u,\xi}[\|F_{\nu}(x, \xi, u) - f_{\nu}(x)\|^2] \leqslant \sigma_{f, \nu}^2$, where $\sigma_{f, \nu}^2 := (\textstyle\sum_{i=1}^m 4(n+2)M_{i}^2\nu_i^2 + L_i^2\nu_i^4n^2) + 2\sigma_f^2$, where $\sigma_f^2 = \sum_{i=1}^m \sigma^2_{f_i}$.
\end{lemma}

\begin{lemma}\label{thm:gradestrate}
Let $\tilde{B}_i := \frac{\nu_i}{2}L_i (n+3)^{3/2} + L_i D_X + M_{i}$. Under assumptions \ref{Assumption1} and \ref{Assumption3}, we have 
\begin{align}
    &\mathbb{E}_{u,\xi}[\| G_{i, \nu_i}(x, \xi, u) - \nabla f_{i, \nu_i}(x) \|^2]  \leqslant  \sigma_{i, \nu_i}^2,\label{eq:1.30}
\end{align}
where $\sigma_{i, \nu_i}^2 \coloneqq    \nu_i^2L_i^2(n+6)^3 + 10(n+4)[\sigma_i^2 + \tilde{B}_i^2]$.
\end{lemma}


\subsection{Algorithmic Methodology}\label{sec:methods}
We now present the \texttt{SZO-ConEX} algorithm for solving the stochastic zeroth-order functional constrained optimization problem~\eqref{eqn:1.1}. \textcolor{black}{The constraint extrapolation framework is a novel primal-dual method that proceeds by (i) considering the Lagrangian formulation of~\eqref{eqn:1.1}, (ii) constructing linear approximations for the constraint functions, and (iii) constructing an \emph{extrapolation operation} which enables acceleration. Such an approach has the advantage that: (i) it does not require the projection of Lagrangian multipliers onto a possibly unknown bounded set (which is required by several other primal-dual methods), (ii) it is a single-loop algorithm with a built-in acceleration step. It is worth remarking that~\cite{boob2019proximal} and~\cite{hamedani2018primal} showed that such an approach helps achieve better rate of convergence than existing methods for solving Lagrangian problems (of the form in~\eqref{eq:1.5} below) in the stochastic first-order setting}. However, their approach is not directly applicable to the zeroth-order setting where the estimated stochastic gradients are biased and have variances that are not uniformly bounded. 

Recall the problem in~\eqref{eqn:1.1} and notice that there are two types of constraints. \textcolor{black}{The set $X$ represents \emph{known} constraints (i.e., constraints that are analytically available)} and the inequality constraints defined by the functions $f_i$, $i\in [m]$ are the \emph{unknown or zeroth-order constraints}. The Lagrangian of~\eqref{eqn:1.1} is given by
\begin{align}
    \min_{x \in X}\max_{y \geqslant \mathbf{0}}\{\mathcal{L}(x, y) := f_0(x) + \textstyle\sum_{i=1}^m~y_i f_i(x)\}. \label{eq:1.5}
\end{align}
In other words, $(x^*, y^*)$ is a \textit{saddle point} of the Lagrange function $\mathcal{L}(x, y)$ such that 
\begin{align}
    \mathcal{L}(x^*, y) \leqslant \mathcal{L}(x^*, y^*) \leqslant \mathcal{L}(x, y^*), \label{eq:1.3}
\end{align}
for all $x \in X, y \geqslant \mathbf{0}$, whenever the optimal dual, $y^*$, exists. Throughout this work, we assume the existence of $y^*$ satisfying \eqref{eq:1.3}. In order to handle the zeroth-order setting, we also define Lagrangian with the smoothed functions as 
\begin{align}\label{eq:zolagrangian}
\mathcal{L}_{\nu}(x, y) \coloneqq f_{0, \nu_0}(x) + \textstyle\sum_{i=1}^m~y_i f_{i, \nu_i}(x).
\end{align}

Now, we describe the linearization in the context of the iterates directly as it will be easier to understand in the stochastic setting that we are in. Let $x^{(t)}$ be the sequence \textcolor{black}{produced by} the algorithm (to be discussed later). The linearization of $f(\cdot)$ at the point $x^{(t)}$, with respect to the point $x^{(t-1)}$, is given by 
\begin{align*}
    \ell_{f}(x^{(t)})\coloneqq f_{\nu}(x^{(t-1)}) + \nabla f_{\nu}(x^{(t-1)})^T(x^{(t)} -x^{(t-1)}),
\end{align*}
where similar to $\nabla f$, we define $\nabla f_{\nu}(x^{(t-1)})\coloneqq [\nabla f_{1, \nu_1}(x^{(t-1)}), \dots, \nabla f_{m, \nu_m}(x^{(t-1)})]$. For the implementation, we use the version of linearization with the Gaussian smoothing based stochastic zeroth-order gradients. In particular, we define 
$\ell_{F}(x^{(t)})\coloneqq F_{\nu}(x^{(t-1)}, \bar{\xi}^{(t-1)}, \bar{u}^{(t-1)})+ {G}_{\nu}(x^{(t-1)}, \overline{\xi}^{(t-1)}, \overline{u}^{(t-1)})^T(x^{(t)} - x^{(t-1)}),$
where ${G}_{\nu}(x^{(t-1)}, \overline{\xi}^{(t-1)}, \overline{u}^{(t-1)}) \in \mathbb{R}^{n \times m}$ is given by $$ \hspace{-0.25in}[G_{1, \nu_1}(x^{(t-1)}, \overline{\xi}_1^{(t-1)}, \overline{u}_1^{(t-1)}), \dots, G_{m, \nu_m}(x^{(t-1)}, \overline{\xi}_m^{(t-1)}, \overline{u}_m^{(t-1)})].$$ Here, by $\overline{\xi}^{(t-1)}, \overline{u}^{(t-1)}$ we mean an independent (of $\xi^{(t-1)}, u^{(t-1)}$, respectively) realization of random objects $\xi, u$, respectively.

Based on this, the overall procedure, termed as~\texttt{SZO-ConEx}~is provide\textcolor{black}{d} in Algorithm~\ref{Algorithm1}. We now explain the individual steps in more detail.
\begin{itemize}
    \item \textcolor{black}{Step 3: This extrapolation step, considered by~\cite{boob2019proximal} (see also,~\cite{hamedani2018primal}) for the stochastic first-order setting forms the main methodological innovation over existing primal-dual method. First, note that instead of working with constraint functions, we work with a stochastic linearization of them. The extrapolation or moving average is essentially a way to incorporate momentum in the $s^{(t)}$ sequence. From the analysis, it turns out that the choice of constant $\theta_t$ (which we set as $\theta_t=1$ without any loss of generality) gives the best possible oracle complexity in our analysis.}
    
    \textcolor{black}{It is also worth remarking that the extrapolation/moving-average approach has been also used recently in stochastic optimization of composition of two functions in~\cite{ghadimi2020single}. Furthermore, the linearization technique is also used in stochastic optimization of composition of $T$ functions, for any $T\geq1$, in \cite{ruszczynski2021stochastic} and \cite{balasubramanian2022stochastic}.   }
    \item \textcolor{black}{Step 4: This step corresponds to the gradient ascent step to address the maximization problem in the Lagrangian formulation. We let parameter $\tau_t$ depend on $t$ in the algorithm. However, the analysis in Section~\ref{sec:thms} reveals that a constant step-size of $\tau_t=\tau$ suffices to obtain the derived oracle complexity.} 
    \item \textcolor{black}{Step 5: This step corresponds to the descent step, or more precisely the proximal gradient descent step to solve minimization part of the saddle point problem in the Lagrangian formulation. We remark that one could potentially replace the proximal gradient step with a conditional gradient step when performing linear-minimization over the set $X$ is computationally efficient. We leave a rigorous oracle complexity analysis of this modification as future work.}
    \item \textcolor{black}{Step 6: This step corresponds to the averaging of the iterates. As we demonstrate later in the analysis in Section~\ref{sec:thms}, in the convex and non-convex settings that we consider, the best oracle complexities obtained correspond to the case of constant choice, i.e., $\gamma_t=1$ without loss of generality. However, we suspect that there might be advantages of considering time-varying $\gamma_t$ for the challenging case of adaptive algorithms, that do not necessarily know the structure of the optimization problem at hand. We leave a detailed analysis of such adaptive algorithms as future work.}  
\end{itemize}

\textcolor{black}{F}inally, it is worth noting that~\cite{gramacy2016modeling} proposed an augmented Lagrangian approach for solving the problem in~\eqref{eqn:1.1} in the non-noisy setting. However,  they did not propose the above constraint extrapolation technique. In our experiments in Section~\ref{sec:experiments}, we show that our constraint extrapolation approach significantly outperforms the approach in \cite{gramacy2016modeling} in simulations and real-world problems.  
\begin{algorithm}[t]
\caption{Stochastic Zeroth-Order Constraint Extrapolation Method (\texttt{SZO-ConEx})}
\begin{algorithmic}[1]
\Require $\nu_0>0$, $\nu > \mathbf{0}, (x^{(0)}, y^{(0)}), \{\gamma_t, \tau_t, \eta_t, \theta_t\}_{t \geqslant 0}, T$.
\State Set $(x^{(-1)}, y^{(-1)}) \gets (x^{(0)}, y^{(0)}), F_{\nu}(x^{(-1)}, \overline{\xi}^{(-1)}, \overline{u}^{(-1)}) \gets F_{\nu}(x^{(0)}, \overline{\xi}^{(0)},
 \overline{u}^{(0)})$, $\ell_F(x^{(-1)}) \gets \ell_F(x^{(0)})$.
\For{$t = 0, \dots, T-1$}
\State $s^{(t)} \gets (1 + \theta_t)\ell_F(x^{(t)}) - \theta_t\ell_F(x^{(t-1)})$.
\State $y^{(t+1)} \gets [y^{(t)} + \frac{1}{\tau_t}s^{(t)}]_+$.
\State $x^{(t+1)} \gets \textbf{prox}\bigg(G_{0,\nu_0}(x^{(t)},\xi_0^{(t)}, u_0^{(t)})+ \textstyle\sum_{i=1}^m G_{i, \nu_i}(x^{(t)},\xi_i^{(t)}, u_i^{(t)})y_{i}^{(t+1)}, x^{(t)}, \eta_t\bigg).$ 
\EndFor
\State \textbf{return} $\bar{x}_T = (\sum_{t=0}^{T-1}\gamma_t)^{-1}\sum_{t=0}^{T-1}\gamma_tx^{(t+1)}$.
\end{algorithmic}
\label{Algorithm1}
\end{algorithm}

\section{Main results}\label{sec:thms}
We now present our main results on the oracle complexity of \texttt{SZO-ConEX} algorithm. Recall the definition of the stochastic zeroth-order gradient estimators from~\eqref{eq:gradest}. At a high-level, the algorithm could be interpreted as using the constraint extrapolation method of~\cite{boob2019proximal} \textcolor{black}{for solving~\eqref{eq:1.5} with $\mathcal{L}(x,y)$ replaced by $\mathcal{L}_\nu(x,y)$ as defined in~\eqref{eq:zolagrangian}}, as the stochastic zeroth-order gradients used in Algorithm~\ref{Algorithm1} are essentially unbiased estimators of the smoothed functions $f_{\nu,i}$ (for $i \in [m]$). However, they have unbounded variance. Hence, the analysis of~\cite{boob2019proximal}, which is for the stochastic first-order setting under the assumption of unbiased stochastic gradient and uniformly bounded variance is not directly applicable. Furthermore, on the one hand as the smoothing parameters $\nu_i$ (for $i \in [m]$) tend to zero, $\mathcal{L}_{\nu}(x, y)$ converges to $\mathcal{L}(x, y)$ defined in~\eqref{eq:1.5}. However, on the other hand, the parameters $\nu_i$ are in the denominator of the stochastic zeroth-order gradient estimators (see~\eqref{eq:gradest}). Hence, we cannot let them tend to zero at any arbitrary rate. Picking the tuning parameters $\nu_i$ carefully to balance this tension and get the best possible oracle complexity forms the crux of our analysis. Finally, we also point out that general strategies for picking the smoothing parameters (as proposed in~\cite{beck2012smoothing} for dealing with non-smooth stochastic first-order optimization problems) are also not  directly applicable for analyzing stochastic zeroth-order algorithms and specialized approaches are often required -- we refer the reader to~\cite{duchi2015optimal, nesterov2017random,ghadimi2013stochastic, balasubramanian2018zeroth} for several related techniques for analyzing unconstrained stochastic zeroth-order optimization algorithms.

\subsection{Convex Setting}\label{sec:convex}
We first provide our theoretical results for the case when the functions $f_i$, for $i \in [m]$, are convex. We start by describing the measure of optimality we consider for solving~\eqref{eqn:1.1}.

\begin{definition}\label{def:epsconvex}
A point $\bar{x}$ is an $\epsilon$-approximately optimal solution in expectation, for~\eqref{eqn:1.1}, if it satisfies $\mathbb{E}[f_0(\bar{x}) - f_0^*] \leqslant \epsilon $ and $\mathbb{E}[\| [f(\bar{x})]_+ \|_2] \leqslant \epsilon,$ where $f_0^*$ is the optimal value of~\eqref{eqn:1.1} and the expectation is with respect to the randomness arising due to $\xi_i$ and $u_i$ across all iterations.
\end{definition}
The first part of the above definition corresponds to the standard optimality condition for the convex problem. The next part corresponds to constraint violation.  Our main result is described next. We define $M_X\coloneqq \sup_{x \in X} \|x\|$. Furthermore, we define $\sigma_\nu\coloneqq [\sigma_{1,\nu_1}, \cdots, \sigma_{m,\nu_m}]$, where $\sigma_{i,\nu_i}$, for $i \in [m]$ are as defined in Lemma~\ref{thm:gradestrate}, $\sigma_{X, f} \coloneqq (\sigma_{f, \nu}^2 + D_X^2\| \sigma_{\nu}\|_2^2)^{1/2}$  (where $\sigma_{f, \nu}^2$ is as defined in Lemma~\ref{lemma:zogradbounds}). 

\begin{theorem} \label{thm11}
Suppose the functions $f_i$, for $i\in[m]$, are convex and satisfy Assumptions~\ref{Assumption1},~\ref{Assumption3} and~\ref{Assumption2}. Define $\mathcal{H}_* \coloneqq (L_{f}D_X\| y^*\|_2 )/2$. Set $y_0 = \mathbf{0}$ and $\{\gamma_t, \theta_t, \eta_t, \tau_t\}$ in Algorithm~\ref{Algorithm1} according to the following: $\gamma_t = 1, \quad \eta_t = L_{0} + L_{f} + \eta,$ and $\theta_t = 1, \quad \tau_t = \tau$,
where 
\begin{align*}
    \eta  := &\max\Bigg\{ \frac{\sqrt{2T[\mathcal{H}_*^2 + \sigma_{0, \nu_0}^2 + 48\|\sigma_{\nu}\|_2^2]}}{D_X}, \frac{6\max\{2M_f, 4\|\sigma_{\nu}\|_2\}}{D_X}\Bigg\},  \\
    \tau  &:= \max\Bigg\{\sqrt{96T}\sigma_{X, f}, 2D_X\max\{M_f, 4\|\sigma_{\nu}\|_2\}\Bigg\}.
\end{align*}
Then, we have 
\begin{align}
\mathbb{E}[f_0(\bar{x}_T) - f_0(x^*)] 
 \leqslant &\frac{(L_{0} + L_{f})D_X^2 + \max\{12M_f, 24\|\sigma_{\nu}\|_2\}D_X}{T}\notag\\
 &+ \frac{1}{\sqrt{T}}\sqrt{2(\mathcal{H}_*^2 + \sigma_{0, \nu_0}^2 + 48\| \sigma_{\nu}\|_2^2)}D_X \nonumber \\
  &+ \frac{1}{\sqrt{T}}\left\{\frac{\sqrt{2} \zeta^2 D_X}{\sqrt{\mathcal{H}_*^2 + \sigma_{0, \nu_0}^2 + 48\|\sigma_{\nu}\|_2^2}} + \frac{\sqrt{3}\sigma_{X, f}}{\sqrt{2}} \right\}\notag \\
  &+[\nu_0^2L_0n + M_Xn(\textstyle\sum_{i=1}^m\nu_i^4L_i^2)^{1/2}], \label{eq:1.61}
\end{align}
and
\begin{align}
 \mathbb{E}[\| [f(\bar{x}_T)]_+ \|_2] \leqslant & \frac{1}{\sqrt{T}}\Bigg\{\left[12\sqrt{6}(\|y^*\|_2 + 1)^2 + \frac{13}{4\sqrt{6}}\right]\sigma_{X, f} \notag \\ 
&  + [\nu_0^2L_0n + M_Xn(\textstyle\sum_{i=1}^m\nu_i^4L_i^2)^{1/2}] \notag \\
     & + \sqrt{2}D_X\Bigg[\sqrt{\mathcal{H}_*^2 + \sigma_{0, \nu_0}^2 + 48\|\sigma_{\nu}\|_2^2} \notag\\&+ \frac{\zeta^2 + \mathcal{H}_*^2}{\sqrt{\mathcal{H}_*^2 + \sigma_{0, \nu_0}^2 + 48\|\sigma_{\nu}\|_2^2}}\Bigg]\Bigg\} \notag \\ & \hspace{-0.3in}+\frac{(L_{0} + L_{f})D_X^2 + \max\{12M, 24\|\sigma_{\nu}\|_2\}D_X\left(1 +(\|y^*\|_2 + 1)^2\right)}{T},  \label{eq:1.62}
\end{align}
where $ \zeta\coloneqq 2e\{\sigma_{0, \nu_0}^2 + \|\sigma_{\nu}\|_2^2(14\|y^*\|_2^2 + 75) + 2\sqrt{3}\|\sigma_{\nu}\|_2(2\mathcal{H}_* + \sigma_{0, \nu_0} + \sqrt{48}\| \sigma_{\nu}\|_2) + \sqrt{6}D_X^{-1}\|\sigma_{\nu}\|_2[\nu_0^2L_0n + M_Xn(\textstyle\sum_{i=1}^m\nu_i^4L_i^2)^{1/2}]\sqrt{T}\}^{1/2}.$ 
Hence, by choosing,  
\begin{align}
\nu_0 & \leqslant \min\left\{\frac{1}{\sqrt{2L_0n\sqrt{T}}}, \frac{2}{(n+3)^{3/2}}, \frac{1}{L_i(n+6)^{3/2}}\right\}\label{nunot} \\
\nu_i & \leqslant \min\Bigg\{\frac{2}{(n+3)^{3/2}}, \frac{1}{2M_{ i}\sqrt{(n+2)m}}, \frac{1}{\sqrt{L_in\sqrt{m}}},\frac{1}{\sqrt{2L_inM_X\sqrt{Tm}}}, \frac{1}{L_i(n+6)^{3/2}\sqrt{m}}\Bigg\},\label{eq:nui}
\end{align}
for $i\in[m]$, the number of calls to the stochastic zeroth-order oracle required by Algorithm~\ref{Algorithm1} to find an $\varepsilon$-approximately optimal solution of~\eqref{eqn:1.1} is of the order \textbf{$\mathcal{O}\left( \left((m+1)n\right)/\epsilon^2\right)$}.
\end{theorem}

\begin{remark}
\textcolor{black}{Although the parameter settings of Theorem~\ref{thm11} and the right hand side of~\eqref{eq:1.61} and~\eqref{eq:1.62} appear complicated to parse, the important take away message is that the right hand side of~\eqref{eq:1.61} and~\eqref{eq:1.62} are of the order $\mathcal{O}(1/\sqrt{T})$ which leads to the oracle complexity described above.} Furthermore, the order of $\epsilon$ in the oracle complexity is of the same order as that in~\cite{boob2019proximal} for the stochastic first-order setting. The $(m+1)n$ factor in the oracle complexity appears because we are required to estimate $m+1$ gradient vectors, each of dimension $n$. The dimension dependency is unavoidable even in the unconstrained setting, as showed via lower bounds in~\cite{jamieson2012query, duchi2015optimal}. For a fixed dimensionality $n$, the oracle complexity in the zeroth-order setting is linear in the number of constraints $m$. 
\end{remark}

\begin{remark} \label{rem:worserate}
A word is in order regarding the choice of the tuning parameters $\nu_i$, $i \in [m]$ in \eqref{eq:nui}. If one follows the standard analysis for selecting the tuning parameters for stochastic zeroth-order algorithms, which are predominantly developed for unconstrained problems, the $m$  related factors appearing in the choice of $\nu_i$ would be missed. This subsequently would lead to an increased dependency of the oracle complexity on $m$, instead of the linear dependency that we obtain now. A main part of our proof involves obtaining the choice of the smoothing parameters $\nu_i$ as in~\eqref{eq:nui}, that helps us to obtain oracle complexity as stated in Theorem~\ref{thm11}.
\end{remark}

\subsection{\textcolor{black}{Proximal-point based Meta-Algorithm for the Nonconvex Setting}}\label{sec:nonconvex}
We now consider the case when objective function $f_0$, and the constraint functions $f_1,\ldots, f_m$ are nonconvex. \textcolor{black}{In this case,~\cite{boob2019proximal}, analyzed a two-step meta-algorithm, which is based on the standard proximal-method; see, for example~\cite{drusvyatskiy2017proximal} for a survey.}  

\textcolor{black}{The basic idea behind the method (as stated in Algorithm~\ref{Algorithm2}) consists of the following two steps: (i) construct a sequence of convex relaxations for the nonconvex problem, and (ii) leverage the algorithm developed for the convex setting. Given our Algorithm~\ref{Algorithm1}, we leverage this framework to solve~\eqref{eqn:1.1} in the nonconvex setting. }
\begin{algorithm}[t]
\caption{Meta-Algorithm for Nonconvex Setting}
\label{Algorithm2}
\begin{algorithmic}[1]
\Require Input $x_0$, parameters $\mu_o$, $\mu_i$, $i \in [m]$.
\For{$k = 1, \dots, K$}
\State For  $i \in [m]$, set:
 \begin{align*} f_0(x; x_{k-1})& \coloneqq f_0(x) + 2\mu_0W(x, x_{k-1}),\\
f_i(x; x_{k-1}) &\coloneqq f_i(x) + 2\mu_iW(x, x_{k-1}).
\end{align*}
\State  Obtain an $\epsilon$-approximately optimal solution to the problem:
\begin{align}
     \arg\min_{x \in X} f_0(x; x_{k-1})  \quad \text{s.t.} \quad f_i(x; x_{k-1}) \leqslant 0, \quad i \in [m].
\end{align}
by using \texttt{SZO-ConEx} in Algorithm~\ref{Algorithm1}. Denote it by $x_k$, for $k=1,\ldots, K$.
\EndFor
\State Randomly choose $\hat{k} \in \{1,\ldots, K \}$
\State \textbf{return} $x_{\hat{k}}$.
\end{algorithmic}
\end{algorithm}

We first define the exact Karush-Kuhn-Tucker (KKT) condition for~\eqref{eqn:1.1}~as follows. For a convex set $X$, we denote \textcolor{black}{its} interior as $\text{int} X$, the normal cone at $x \in X$ as $N_X(x)$, and its dual cone as $N_X^*(x)$. \textcolor{black}{For convenience, we recall the definition of normal cone: For convex set $X$, we have $N_X^*(x)\coloneqq \{ v\in \mathbb{R}^n: \langle y, z-x  \rangle \leq 0,~\text{for all}~z \in X \}$; see~\cite[Part I and II]{rockafellar2015convex} for additional properties and examples}.  Let $\oplus$ denote the Minkowski sum of two sets $A, B \subset \mathbb{R}^n$, \textcolor{black}{defined as $A \oplus B = \left\{ a+b: a\in A~\text{and}~b \in B \right\}$}. We refer to the distance between two sets $A, B \subset \mathbb{R}^n$ as $d(A, B) \coloneqq \inf_{a \in A, b \in B}\|a - b\|$.
\begin{definition}\label{def:epsnonconvex}
We say that $x^* \in X$ is a critical KKT point of~\eqref{eqn:1.1} if $f_i(x^*) \leqslant 0$ and $\exists y^*\coloneqq [y^*_1, \ldots, y^*_m]^T \geqslant \mathbf{0}$ such that 
\begin{align*}
   & y^{*}_i f_i(x^*) = 0, \quad  i \in [m],\\
    d(\nabla f_0(x^*) &+ \textstyle\sum_{i=1}^my^{*}_i\nabla f_i(x^*) \oplus N_X(x^*), \mathbf{0})  = 0.
\end{align*}
\end{definition}

The parameters $\{y^*_i\}_{i \in [m]}$ are called \textit{Lagrange multipliers}. For brevity, we use the notation $y^*$ and $[y^{*}_1, \ldots, y^{*}_m]^T$ interchangeably. With this definition, we also have the following approximate KKT condition which is the standard approximate optimality condition for solving~\eqref{eqn:1.1} in the nonconvex setting. 
\begin{definition}
We say that a point $\hat{x} \in X$ is an $(\varepsilon, \delta)$-KKT point in expectation for~\eqref{eqn:1.1} if there exists $(\bar{x},\bar{y})$ such that $f(\bar{x}) \leqslant \mathbf{0}, \bar{y} \geqslant \mathbf{0}$ and 
\begin{align*}
    \mathbb{E}[\textstyle\sum_{i=1}^m|\bar{y}_i f_i(\bar{x})|] & \leqslant \varepsilon, \mathbb{E}[\| \bar{x} - \hat{x} \|^2 ] \leqslant \delta\\
    \mathbb{E} [(d(\nabla f_0(\bar{x}) + \textstyle\sum_{i=1}^m &\bar{y}_i\nabla f_i(\bar{x}) \oplus N_X(\bar{x}), \mathbf{0}))^2] \leqslant \varepsilon.
\end{align*}
\end{definition}

\begin{proposition}\label{prop:nonconvex}
Consider solving~\eqref{eqn:1.1} with both the objective and the constraint function being nonconvex and satisfying Assumptions~\ref{Assumption1},~\ref{Assumption3} and~\ref{Assumption2}. Then, by running  Algorithm~\ref{Algorithm2} with $K = \mathcal{O}(1/\epsilon)$, we obtain $(\epsilon, 2\epsilon/2\mu_0 \mu_{\max})$-KKT point, where  $\mu_{\max}\coloneqq \max\{\mu_1, \ldots, \mu_m\}$. Hence, the total number of calls to the stochastic zeroth-order oracle is given by $\mathcal{O}\left( \left((m+1)n\right)/\epsilon^3 \right)$.
\end{proposition}
The proof of the above proposition follows immediately by Theorem~\ref{thm11} and Corollary 3.19 from~\cite{boob2019proximal} and is hence omitted. The parameters $\mu_0$ and $\mu_i$, $i\in[m]$ in Algorithm~\ref{Algorithm2} \textcolor{black}{are set according to} the desired level of accuracy based on Proposition~\ref{prop:nonconvex}. To the best of our knowledge, we are not aware of a non-asymptotic result on the oracle complexity of stochastic zeroth-order optimization with stochastic zeroth-order functional constraints, in both the convex and nonconvex settings.

\subsection{\textcolor{black}{Detailed Comparison to~\cite{boob2019proximal}}}
\textcolor{black}
{In this subsection, we highlight the main differences between our work and \cite{boob2019proximal}. As mentioned previously, our methodological and theoretical results builds upon the work of~\cite{boob2019proximal}.} 
\begin{itemize}
    \item \textcolor{black}{\textbf{Methodological:} At a methodological level, our work focuses on the case when we only have noisy function evaluations, whereas~\cite{boob2019proximal} focus on the case when we have access to noisy gradients. To deal with this, we use the Gaussian smoothing based zeroth-order gradient estimator in combination with the constraint extrapolation technique from~\cite{boob2019proximal}.}
    \item \textcolor{black}{\textbf{Biased gradients:} The use of the Gaussian smoothing based zeroth-order gradient estimator leads to stochastic gradients that are biased. Although~\cite{boob2019proximal} consider noisy gradients, they assume their stochastic gradients are unbiased. This complicates the analysis of the zeroth-order setting we work with.}
    \item \textcolor{black}{\textbf{Non-uniform variance:} Apart from the unbiased stochastic gradient assumption, \cite{boob2019proximal} require the variance of their stochastic gradient to be \emph{uniformly bounded} over the entire parameter space. However, the Gaussian smoothing based gradient estimator does not satisfy this assumption. A major technical part of our analysis involves dealing with stochastic gradients that are not uniformly bounded.}
    \item \textcolor{black}{\textbf{Smoothing parameters}: Our method requires dealing with the additional tuning parameters ($\nu_i$'s) that determine the level of smoothing in the zeroth-order gradient estimator. Dealing with this requires a careful analysis, as otherwise one would end up with worser oracle complexity than we have established in this work; see Remark~\ref{rem:worserate} for details. In contrast, \cite{boob2019proximal} do not require dealing with any tuning parameters for their stochastic gradient, due to their generic set of assumptions.}
    \item \textcolor{black}{\textbf{Experiments:} \cite{boob2019proximal} do not provide any experimental verification of their algorithm. In contrast, in Section~\ref{sec:experiments} that follows, we provide a detailed experimental evaluation, comparing to the existing state-of-the-art methods for constrained zeroth-order optimization, and demonstrate the advantages of the proposed approach.}
\end{itemize}

\section{Experimental Results}\label{sec:experiments} 
\textcolor{black}{
We compare the performance of our algorithm (Algorithm~\ref{Algorithm1}) with the following widely used algorithms for constrained zeroth-order optimization.  }
\begin{itemize}
    \item \textcolor{black}{\texttt{ALBO} method by~\cite{gramacy2016modeling}: This method takes a hybrid approach for constrained zeroth-order optimization, based on combining Bayesian optimization (i.e.,  Gaussian process based approaches) with Augmented Lagrangian methods. Specifically, the objective function of Augmented Lagrangian (which is similar in spirit to~\eqref{eq:1.5}) is estimated using Gaussian process priors. This method has various tuning parameters which makes the implementation a bit difficult. In fact,~\cite{gramacy2016modeling} do not provide the full implementation details and mention that ``\emph{many specifics have been omitted for space considerations}". We use the implementation provided in~\cite{gramacy2016lagp} as recommended by~\cite{gramacy2016modeling}.  }  
    \item \textcolor{black}{\texttt{Slack-AL} method by~\cite{picheny2016bayesian}: This method builds upon the~\texttt{ALBO} method and is also a hybrid method. Specifically, a particular step in estimating the objective function using Gaussian process technique (referred to as the Expected-Improvement step) is avoided by using slack variables. Similar to previous mehtod, we use the implementation provided in~\cite{gramacy2016lagp}.}
    \item \textcolor{black}{\texttt{ADMMBO} method by~\cite{ariafar2019admmbo}: This method is also a hybrid method that uses Bayesian optimization methods. However, they use an ADMM-based approach to solve the augmented Lagrangian problem. We follow the recommendation in Section 5.1 of~\cite{ariafar2019admmbo} for the implementation.} 
    \item \textcolor{black}{\texttt{PESC} method by~\cite{hernandez2015predictive}: This is a purely Bayesian optimization method that uses predictive entropy search for solving constrained zeroth-order optimization methods. As mentioned in~\cite{hernandez2015predictive}, ``\emph{One disadvantage of PESC is that it is relatively difficult to implement}". Furthermore, all the implementation details are not provided in detail in~\cite{ariafar2019admmbo}. Hence, we follow the implementation provided in~\cite{ariafar2019admmbo} for our experiments. 
    }
\end{itemize}

Compared to the above methods, our algorithm comes with a theoretical guarantee for setting the various tuning parameters of the proposed algorithm.  
 
We first report simulation experiments on: (i)  the oracle complexity of \texttt{SZO-ConEX} on 2 different test case objective and constraint functions, and (ii) the effect of the smoothing parameters (corresponding to the zeroth-order gradient estimation process) on the oracle complexity. For our experiments, we consider the following optimization problem (termed as Quadratically Constrained Quadratic Programing (QCQP) in the literature) where the objective function and the constraint function are quadratic functions:
\begin{align*}
\min_{x \in \mathbb{R}^n}~& f_0(x)\coloneqq x^\top A_0 x + b_0^\top x +c_0\\
 \text{such that}~& f_1(x)\coloneqq  x^\top A_1 x + b_1^\top x +c_1 \leqslant 1\color{black}{.}
\end{align*}
Here, $A_0,A_1\in\mathbb{R}^{n \times n}$, $b_0, b_1 \in \mathbb{R}^n$, and $c_0,c_1 \in \mathbb{R}$. When the matrices $A_0,A_1\in\mathbb{R}^{n \times n}$ are further assumed to be symmetric and positive semidefinite, the above problem is a convex optimization problem with convex constraints. In the general case, nonconvex QCQPs form a rich class of optimization problems. For example, every polynomial optimization problem with polynomial constraint could be turned into a nonconvex QCQP at the expense of increasing the \textcolor{black}{number of the optimization variables}~\textcolor{black}{\citep{d2003relaxations}}. Furthermore, it is also known that it is NP-hard to find global minimizers of nonconvex QCQP problem in the worst case.

\textbf{Convex setting:} We first consider the convex setting. Here, we set $A_0$ and $A_1$ to be random but fixed symmetric positive semidefinite matrices. Similarly $b_0, b_1, c_0$ and $c_1$ were generated randomly but fixed. Hence, the problem instance is fixed. In our experiments, we only use (noisy) function evaluations of both the objective and constraint functions. We used standard normal distribution and student $t$-distribution with degrees of freedom 5 for the noise in the function evaluations. For Algorithm~\ref{Algorithm1}, $\theta_t$ was set to 1 based on the theoretical result. Furthermore, $\tau$ and $\eta$, the parameters corresponding to the ascent step and the descent step were set based on \textcolor{black}{trial} and error to achieve the best performance. We remark that one could potentially use principled approaches like line-search for setting the step-size parameters~\textcolor{black}{\citep{berahas2019global}}. As we are working in the zeroth-order setting, in our experimentation we provide additional attention to the smoothing parameters ($\nu_0$ and $\nu_1$) corresponding to the zeroth-order gradient estimators . We set them both to $0.05$, $0.1$ and $2$ and report our performance. 

In figure~\ref{fig:bigsimconvex}, we report the function value difference (corresponding to Theorem~\ref{thm11})  versus number of calls to the (noisy) zeroth-order oracle, for various algorithms and our algorithm with the three choices of smoothing parameters. We work with dimensions $n=200$ and $n=500$ for our problem. Note here that it is easy to obtain the function value at the optimal solution for convex QCQP by using standard solvers (we use \texttt{cvxpy} to calculate it). The curves in figure~\ref{fig:bigsimconvex} correspond to average over 100 \textcolor{black}{trials}. We notice that the performance of our algorithm is uniformly better than the compared algorithms in terms of number of function calls required to obtain a prescribe accuracy. Furthermore, we notice that our algorithm is robust to the choice of smoothing parameters: as long as it is small enough, we have fast convergence, but the iterates diverge when the smoothing parameter value is large. 
\begin{figure*}[t]
\centering
\includegraphics[scale=0.41]{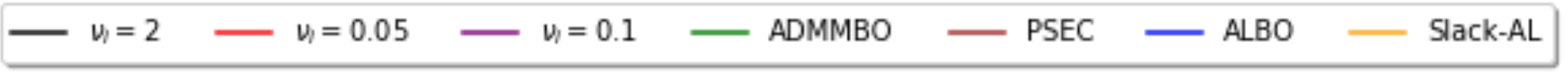}
\includegraphics[scale=0.48]{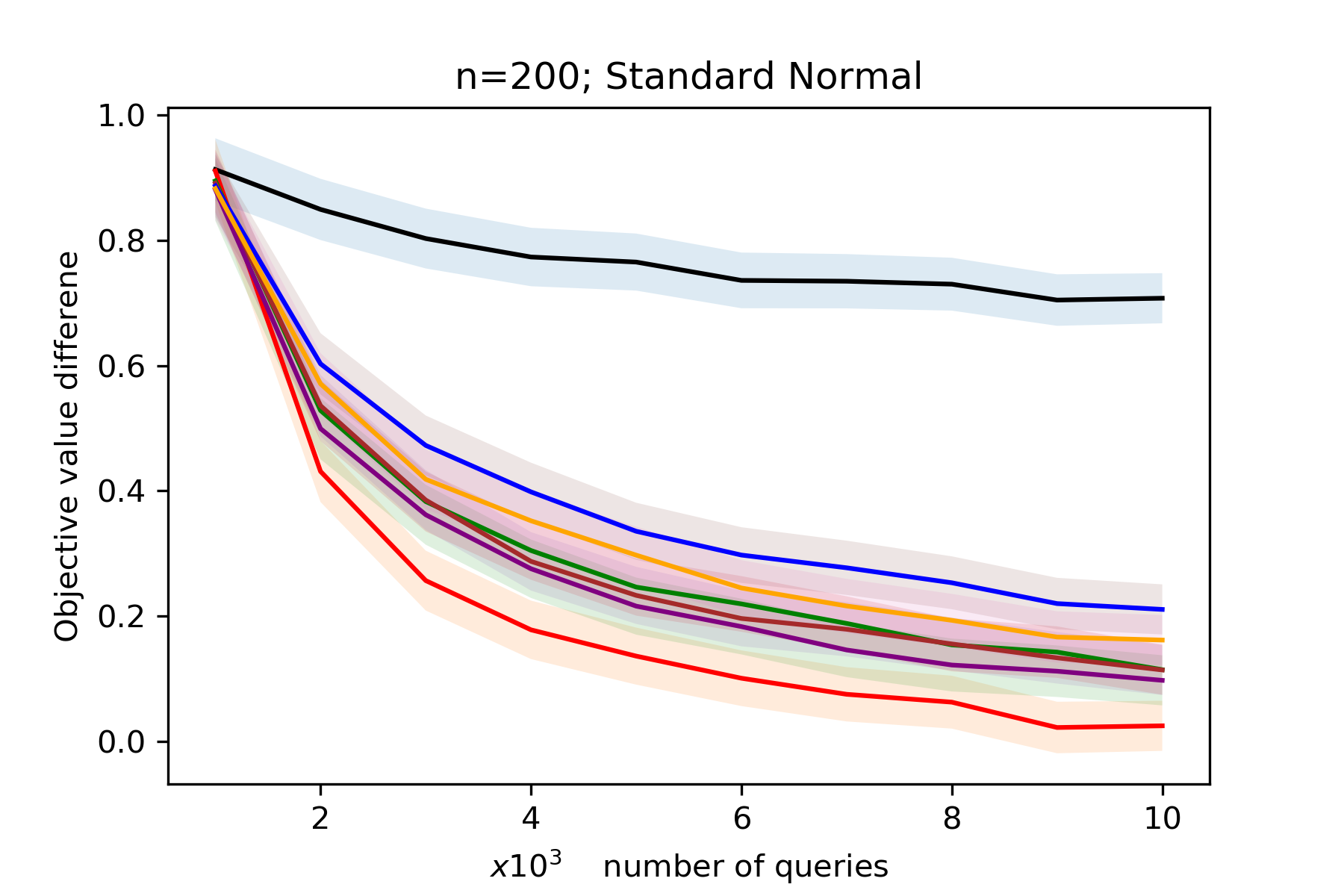}
\includegraphics[scale=0.48]{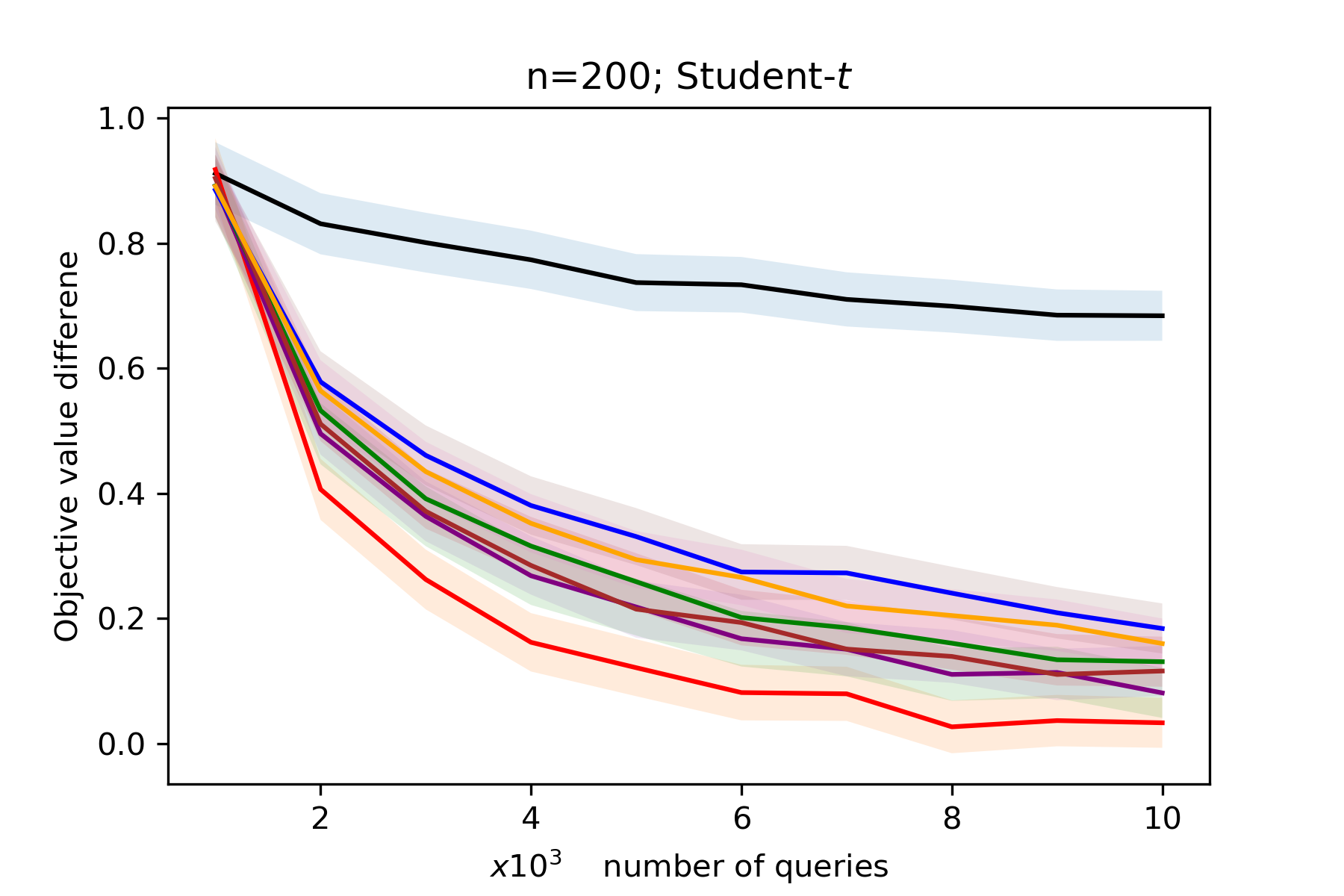}\\
\includegraphics[scale=0.48]{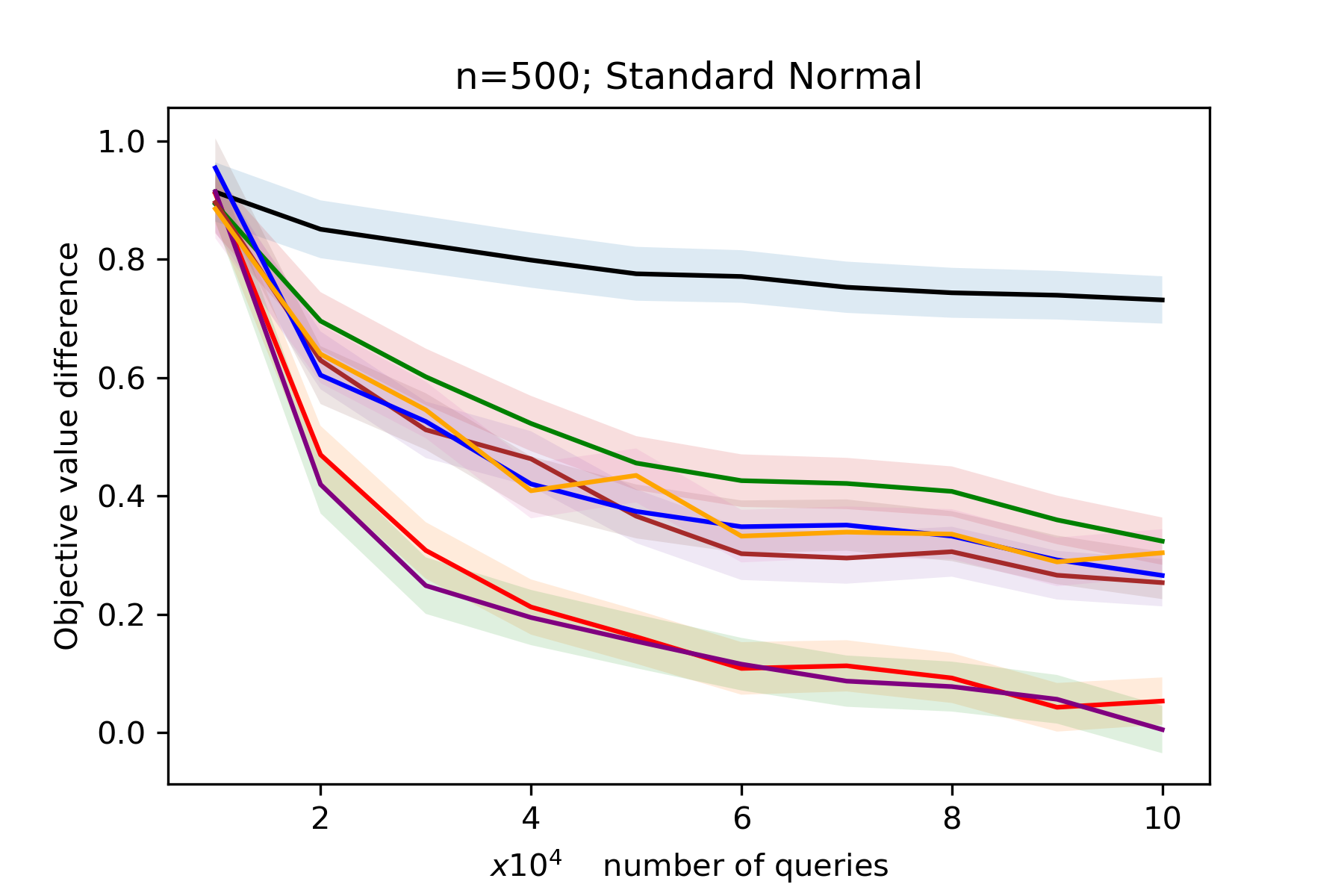}
\includegraphics[scale=0.48]{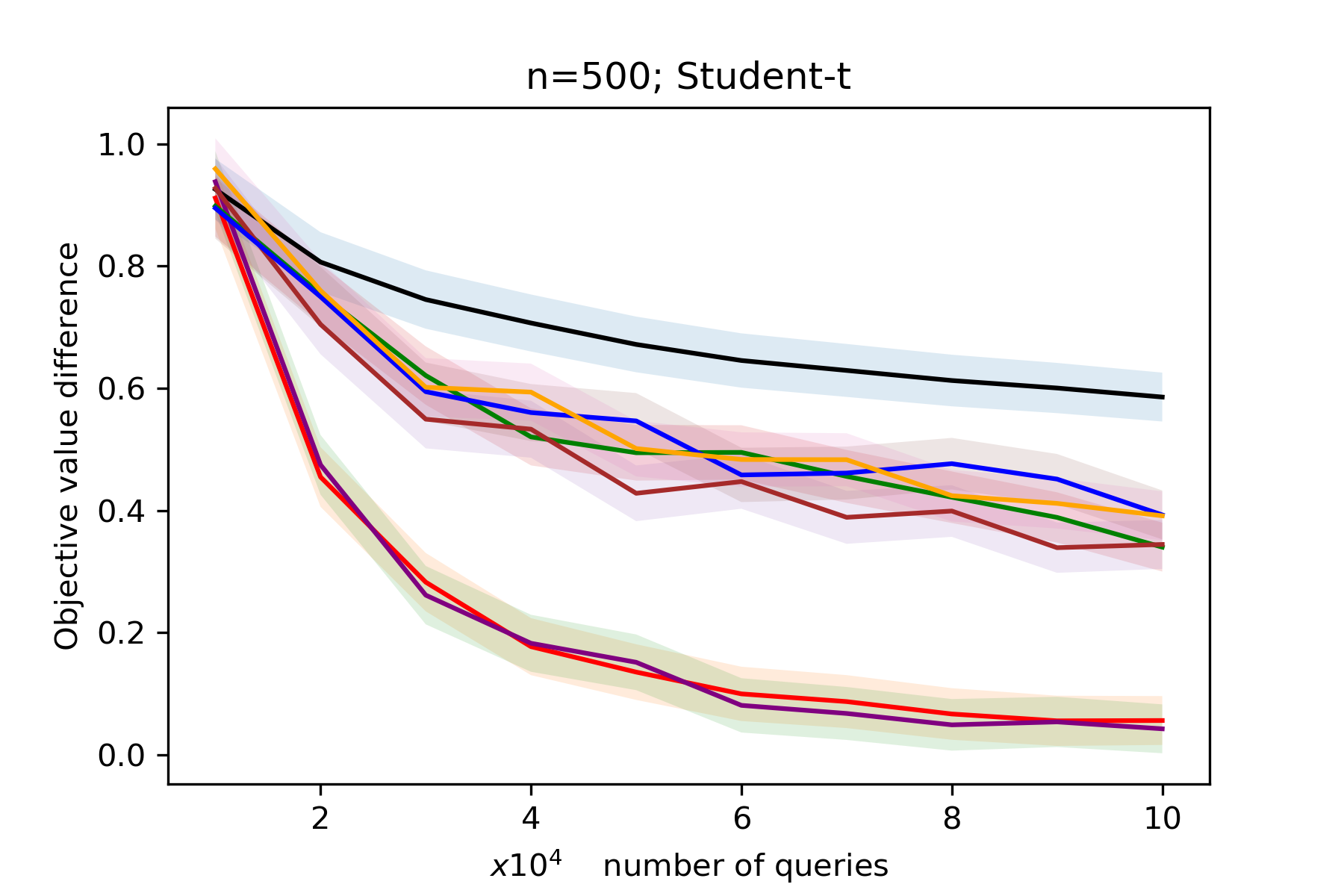}
\caption{Performance comparison on simulation experiment: Plot of number of queries versus objective value difference. The plots represent average curves over 100 trials and the shaded region corresponds to the standard errors. In the legend the \textcolor{black}{curves} corresponding to $\nu_i$ correspond to \texttt{SZO-ConEX} algorithm.}
\label{fig:bigsimconvex}
\end{figure*}

\textbf{Nonconvex setting:} We now list the changes we make for the nonconvex setting. First, while the matrices are still random but fixed, we make them non-positive-definite. Furthermore, for Algorithm~\ref{Algorithm2}, we set $K=50$. In figure~\ref{fig:bigsimnc} (bottom two rows), we report the norm of the gradient of the objective function (corresponding to Theorem~\ref{prop:nonconvex}) versus number of calls to the (noisy) zeroth-order oracle, for various algorithms and our algorithm with the three choices of smoothing parameters. The curves in figure~\ref{fig:bigsimnc} correspond to average over 100 \textcolor{black}{trials}. We notice that  similar to the convex case, the performance of our algorithm is uniformly better than the compared algorithms in terms of number of function calls required to obtain a prescribe accuracy. 

\begin{figure*}[t]
\centering
\includegraphics[scale=0.41]{legend.png}
\includegraphics[scale=0.48]{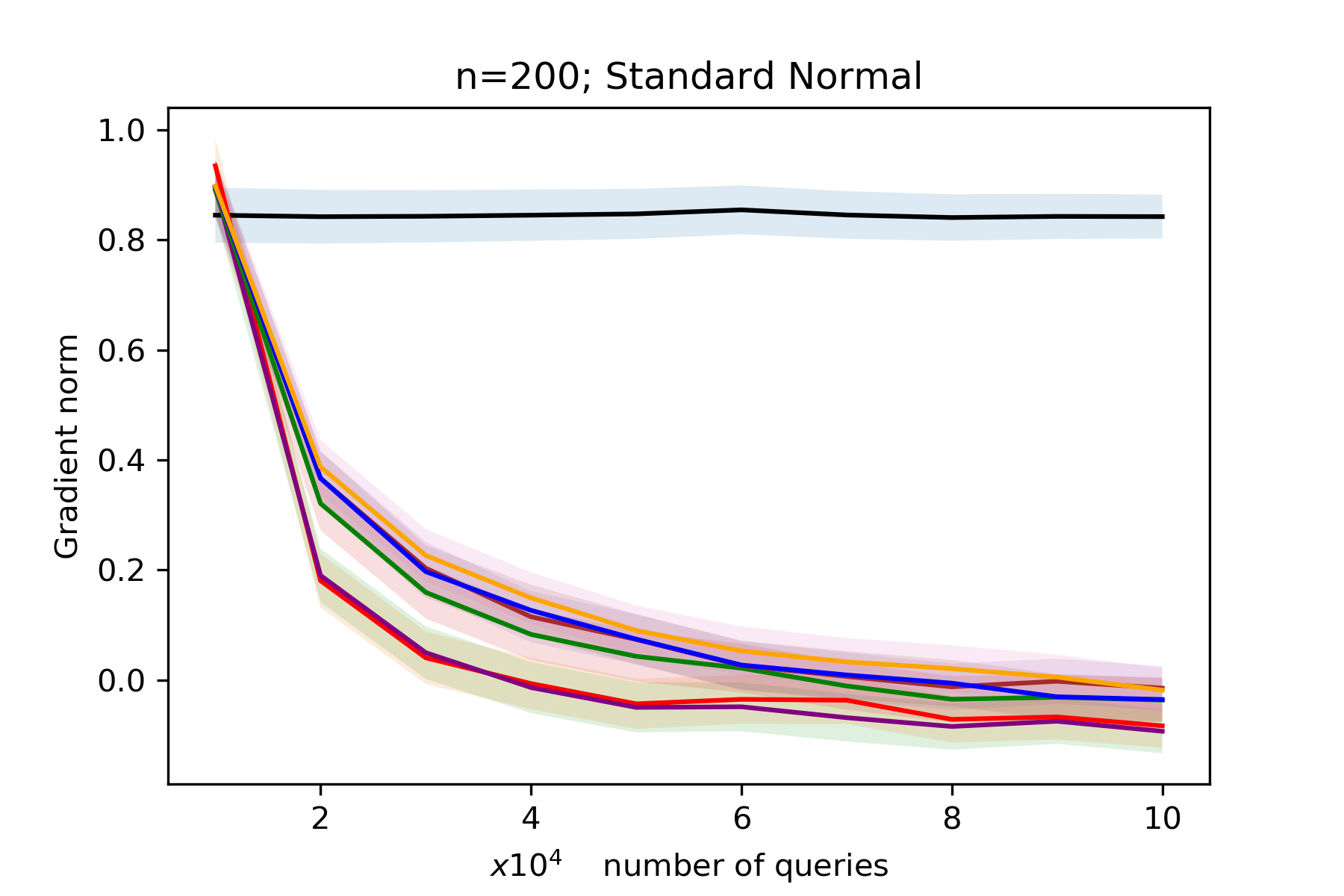}
\includegraphics[scale=0.48]{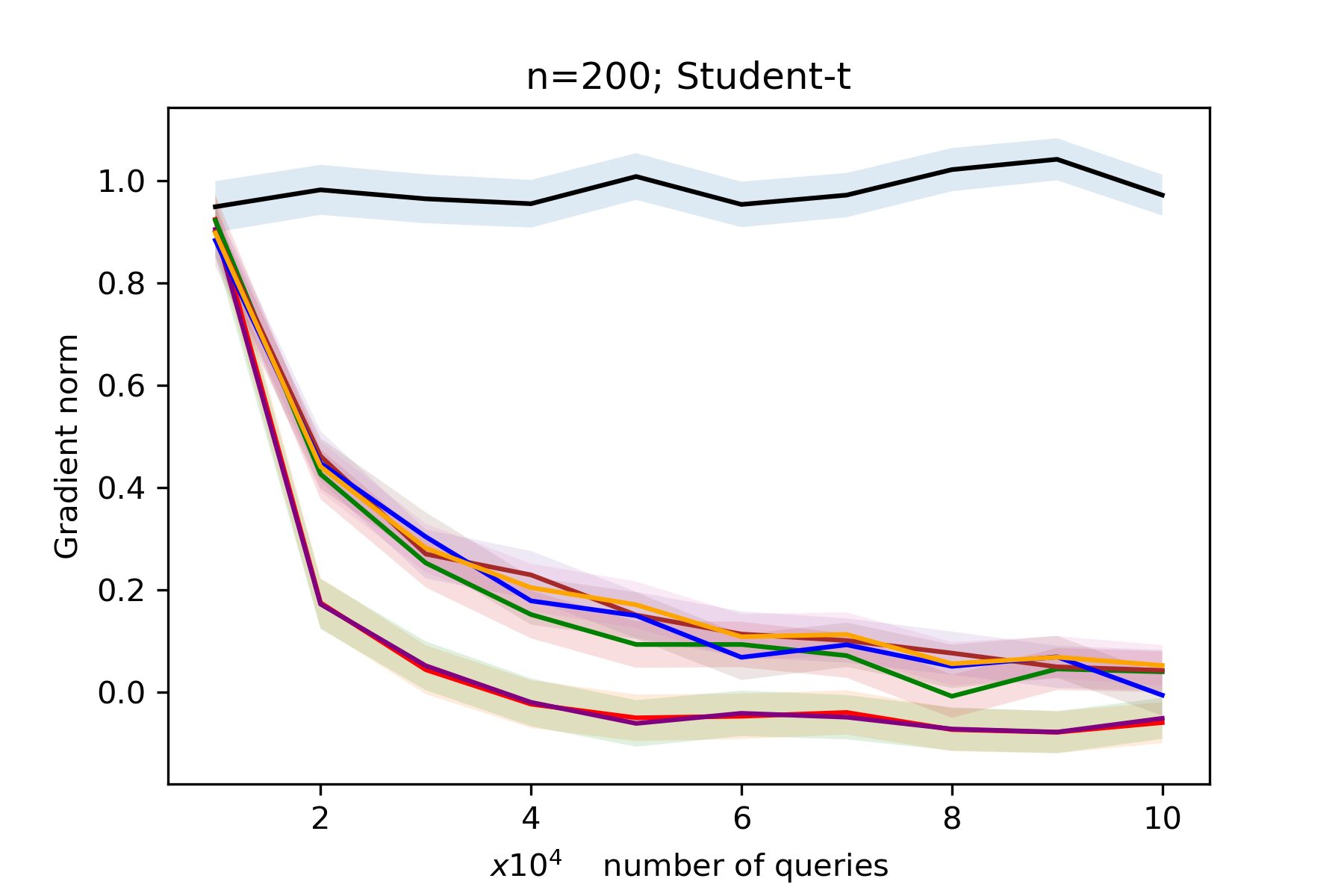}\\
\includegraphics[scale=0.48]{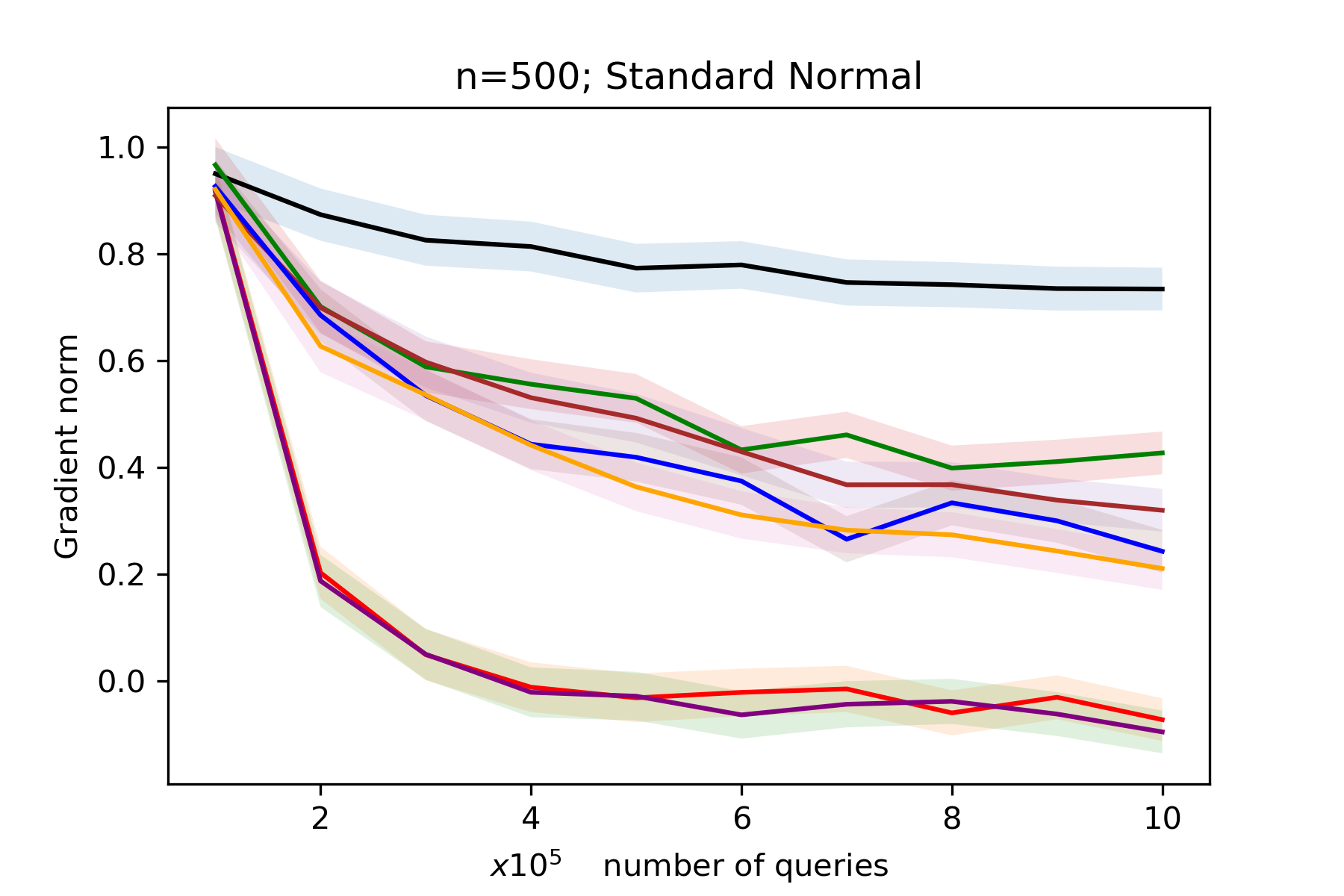}
\includegraphics[scale=0.48]{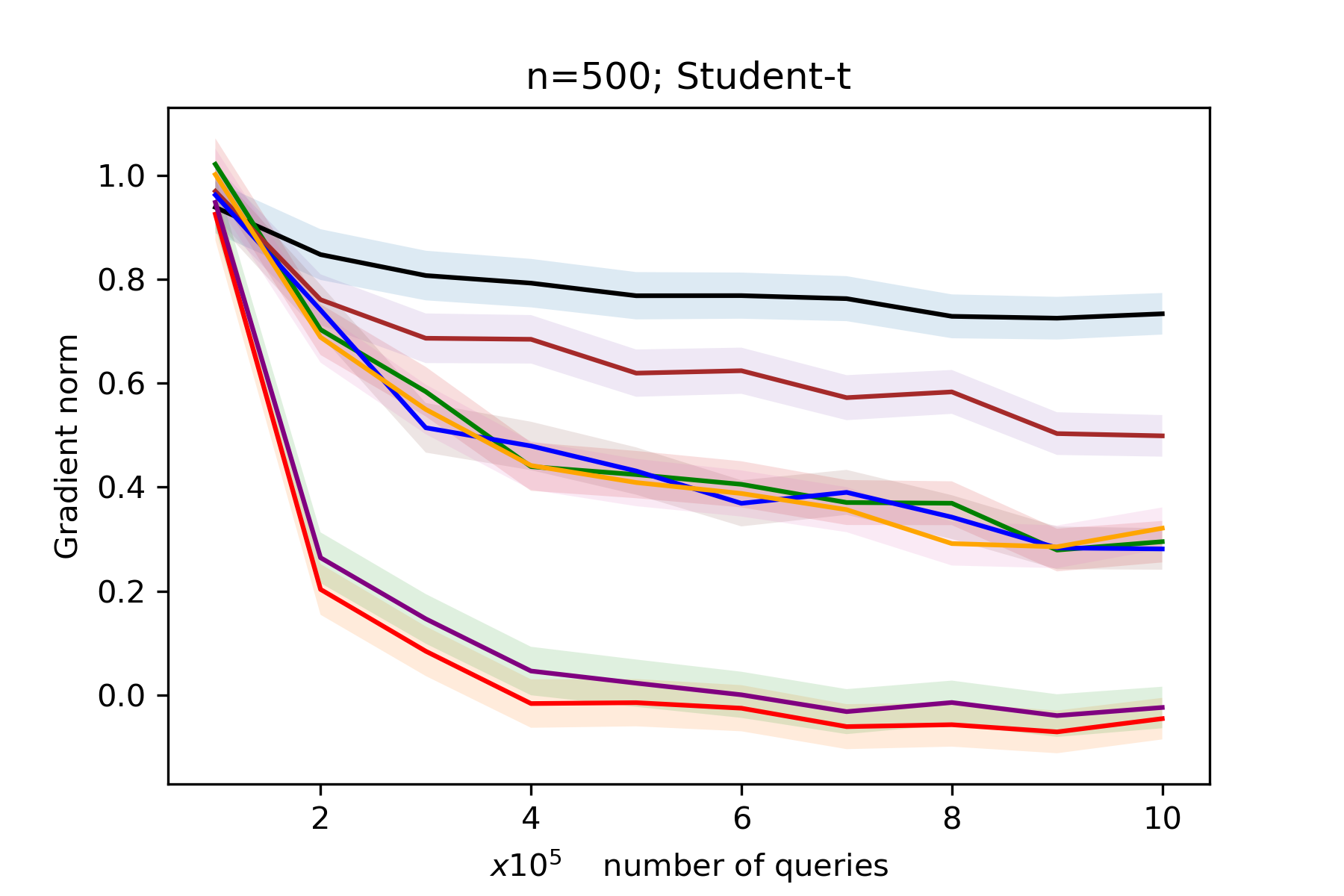}
\caption{Performance comparison on simulation experiment: Plot of number of queries versus norm of the gradient. The plots represent average curves over 100 trials and the shaded region corresponds to standard error. In the legend the \textcolor{black}{curves} corresponding to $\nu_i$ correspond to \texttt{SZO-ConEX} algorithm.}
\label{fig:bigsimnc}
\end{figure*}

A brief summary of the observations are: (i) the oracle complexity of \texttt{SZO-ConEX} method is consistently lower than other existing techniques including \texttt{ALBO}~\cite{gramacy2016modeling}, \texttt{Slack-AL}~\cite{picheny2016bayesian}, \texttt{ADMMBO}~\cite{ariafar2019admmbo}, and \texttt{PESC}~\cite{hernandez2015predictive}, highlighting the benefit of  \emph{constraint extrapolation} step, and (ii) the \texttt{SZO-ConEX} method is robust to the smoothing parameters as long as it is less than a particular threshold. Next, we report the performance of our algorithm on the two motivating examples from Section~\ref{sec:intro}.

\subsection{Application I: Tuning HMC Algorithm}\label{sec:hmctuning}
We now consider the problem of optimizing the hyperparameters of the HMC algorithm. A brief description of the HMC algorithm is provided in Section~\ref{sec:HMCbasics} for completeness. We follow~\cite{gelbart2014bayesian, hernandez2015predictive} closely for the experimental setup. The specific hyperparameters that we consider for this experiment are: (i) the number of leapfrog steps, denoted by $\tau$, (ii) step-size parameter, denoted by $\eta$, (iii) scalar coefficient of the mass matrix, denoted by $\kappa$ (here, following~\cite{neal2011mcmc}, we parametrize the mass matrix as $\kappa$ times an identity matrix\textcolor{black}{)}, and (iv) the fraction of the allotted time the algorithm spends in the burning phase. Hence, the \textcolor{black}{optimization variables are given by $x\in\mathbb{R}^4$. We remark that while the number of leap-frog steps is an integer, for our experiments, we consider it to be real-valued number. In practice, we round it off to the closest integer, with ties broken randomly.}  

The objective function we maximize is the number of effective samples in a fixed computation time. This is a widely used diagnostic metric for measuring the performance of sampling algorithms in Bayesian statistical machine learning~\textcolor{black}{\citep{kass1998markov, lenth2001some}}. \textcolor{black}{For sampling problems, effective sample size is defined as follows. First note that the samples outputted by a sampling algorithm are typically correlated. The effective sample size is defined as the number of \emph{independent} samples from the target density that achieves the same performance as the correlated samples outputted by the sampling algorithm.} However, there is no closed-form analytical relationship between this performance measure and the optimization variable $x$. For our experiments, we use the CODA package~\textcolor{black}{\citep{plummer2006coda}} for calculating the effective sample size. The constraint functions that we use are: (i) the generated samples must pass the Geweke diagnostics~\cite{geweke1991evaluating}; the worst Geweke test score across all variables and chains could be at most 2.0, (ii) the generated samples must pass the Gelman-Rubin convergence diagnostics~\textcolor{black}{\citep{gelman1992single}}; the worst Gelman-Rubin score between variables and chains could be at most be 1.2. The analytical form of the above convergence diagnostics and the optimization variable $x$ is also not available in closed-form. We use PyMC package~\textcolor{black}{\citep{patil2010pymc}} for evaluating the above diagnostic metrics.

We tune the HMC sampling algorithm with the above setup for the problem of sampling from the posterior distribution of a logistic regression binary classification problem on the German credit data set from UCI machine learning repository~\textcolor{black}{\citep{Dua:2019}}. The data set contains 1000 observations that are normalized to have unit variance. We initialize each chain randomly with independent draws from a Gaussian distribution with mean zero and standard deviation $10^{-3}$. For each set of inputs, we compute two chains, each with 5 minutes of computation time. As mentioned previously, all our simulation settings are following that of~\cite{gelbart2014bayesian, hernandez2015predictive}. We conduct our experiments by sub-sampling data sets of size 800 from the original dataset and repeating the procedure for 100 \textcolor{black}{trials}. We compare the performance of our algorithm (with $K=50$) with that of \texttt{ALBO} method by~\cite{gramacy2016modeling}, \texttt{Slack-AL} method by~\cite{picheny2016bayesian}, \texttt{ADMMBO} method by~\cite{ariafar2019admmbo}, and \texttt{PESC} method by~\cite{hernandez2015predictive}. The tuning parameters of the respective methods were set according to the guidelines provide\textcolor{black}{d} in the papers. For our algorithm, we found the performance was robust to the choice of the smoothing parameters, as long as it was \textcolor{black}{sufficiently small}. For the performance reported in Table~\ref{tab:hmctuning1}, we set it to $\nu_i=0.05$. In \textcolor{black}{Table~\ref{tab:hmctuning1}}, we report the average Effective Sample Size (ESS) for the various methods, along with the standard deviation. We notice that the performance of~\texttt{SZO-ConEx} is significantly better than that of the other methods, thereby demonstrating the effectiveness of our method for the problem of hyperparameter tuning for HMC sampling algorithm.
\begin{table*}[t]
\centering
\begin{tabular}{|c|c|c|c|c|c|} 
 \hline
 \textbf{Algorithm} &  \texttt{ALBO} & \texttt{Slack-AL}  & \texttt{ADMMBO}& \texttt{PESC}  & \texttt{SZO-ConEx}   \\ \hline  \hline
 \textbf{ESS}  &$9.4\times10^4 \pm 924$ & $9.3 \times 10^4 \pm 982$  &$9.4 \times 10^4 \pm 884$ &  $9.9 \times 10^4 \pm 998$ 
 &  $10.8 \times 10^4 \pm 992$ \\ \hline
\end{tabular}
\caption{Effective Sample Size (ESS) of Hamiltonian Monte Carlo sampling algorithm tuned by various methods, along with their standard error.}
\label{tab:hmctuning1}
\end{table*}
\begin{table*}[t]
\centering
\begin{tabular}{|c||c|c|c|c|c|} 
 \hline
 \textbf{Algorithm} &  \texttt{ALBO} & \texttt{Slack-AL}  & \texttt{ADMMBO} & \texttt{PESC}  & \texttt{SZO-ConEx}   \\ \hline  \hline
 \textbf{VE} on MNIST  &$3.4 \pm 0.05$ & $3.1 \pm 0.08$  &$3.0  \pm 0.05$ &  $2.9 \pm 0.03$ &  $1.9 \pm 0.04$ \\ \hline
 \textbf{VE} on CIFAR-10  &$4.7\pm 0.02$ & $4.0 \pm 0.03$  &$3.9  \pm 0.05$ &  $3.4  \pm 0.03$ &  $2.2 \pm 0.02$ \\ \hline
\end{tabular}
\caption{Validation Error (VE) along the standard error of 3-layer neural network trained using SGD with momentum for 5000 iterations on MNIST and CIFAR-10 datasets by picking hyperparameters tuned by various methods. The numbers reported are \textcolor{black}{related} to the constraint that the prediction time is not greater than 0.050 seconds on a Nvidia Tesla K20 GPU. }
\label{tab:hmctuning2}
\end{table*}

\subsection{Application II: Tuning a 3-Layer Neural Network}\label{sec:dnntuning}
Next, we turn to the problem of tuning the hyperparameters of \textcolor{black}{a} 3-layer neural network with ReLU activation function trained by stochastic gradient descent algorithm with momentum~\textcolor{black}{\citep{sutskever2013importance}} for 5000 iterations. We follow~\cite{hernandez2015predictive, ariafar2019admmbo} closely for the experimental setup.  The specific hyperparameters that we consider for this experiment are: (i) two learning rate parameters (initial and decay rate), (ii) momentum parameters (initial and final), (iii) dropout parameters (input layer and hidden layers), (iv)  regularization parameters corresponding to weight decay and max weight norm, and (v) the number of hidden units in each of the 3 hidden layers. Hence, the \textcolor{black}{optimization variables are given by $x\in \mathbb{R}^{11}$. Similar to the previous experiment, we treat the number of hidden layers as a real-valued variable and use the same rounding technique in practice.}  

The objective function we minimize is the classification error on the validation set (which we call Validation Error (VE)). Indeed, there is no good closed form expression connecting the above mentioned hyperparameters and the VE. The constraint function that we use is that the prediction time must not exceed 0.050 seconds. Here, we compute the prediction time as the average time of
1000 predictions, over a batch of size 128~\textcolor{black}{\citep{hernandez2015predictive, ariafar2019admmbo}}. The number \emph{0.050 seconds} is set based on the computing resource we use (Nvidia Tesla K20 GPU) so that we can see an active trade off between the objective function (the VE) and the constraint function (prediction time). As highlighted by~\cite{hernandez2015predictive, ariafar2019admmbo}, this specific choice is highly dependent on the computing resource used. Clearly, there is no analytical form for the function describing the relationship between the hyperparameters and the constraint function. All our implementations for this experiment were based on PyTorch open source machine learning library~\textcolor{black}{\citep{paszke2019pytorch}}. 

We tune the SGD algorithm with momentum with the above setup for the problem of classification on MNIST~\textcolor{black}{\citep{lecunmnisthandwrittendigit2010}} and CIFAR-10 datasets~\textcolor{black}{\citep{krizhevsky2009learning}}.  For both datasets, we conduct our experiments by sub-sampling $90\%$ of the training data and report our error over 100 \textcolor{black}{trials}. Similar to the previous case, we compare the performance of our algorithm (with $K=50$) with that of \texttt{ALBO} method by~\cite{gramacy2016modeling}, \texttt{Slack-AL} method by~\cite{picheny2016bayesian}, \texttt{ADMMBO} method by~\cite{ariafar2019admmbo}, and \texttt{PESC} method by~\cite{hernandez2015predictive}. The tuning parameters of the respective methods were set as suggested in the respective papers. The smoothing parameter for our algorithm was set as $\nu_i=0.03$. In Table~\ref{tab:hmctuning2}, we report the validation error achieved such that the constraint on the prediction time is respected for the various algorithms. From the results, we notice that the \texttt{SZO-ConEX} method outperforms the other methods on both the MNIST and CIFAR-10 datasets.

\section{Conclusion} \label{sec:conc}

In this paper, we proposed and analyzed stochastic zeroth-order optimization algorithms for nonlinear optimization problems with functional constraints. We consider the case when both the objective function and the constraint functions are observed only via noisy function queries. Our algorithm is based on leveraging the constraint extrapolation technique proposed by~\cite{boob2019proximal} and the Gaussian smoothing technique. We characterize the oracle complexity of the proposed algorithm in both the convex and nonconvex setting. We also apply our methodology \textcolor{black}{to} the problem of hyperparameter tuning for the HMC algorithm and 3-Layer neural networks trained using SGD with momentum, and demonstrate its superior performance. 

For future work, we plan to develop parallel versions of our algorithm for the case when the objective functions and the constraint functions are available only locally in different machines. We also plan to develop lower bounds on the oracle complexity of stochastic zeroth-order optimization algorithms in the constrained setting. It is of great interest to find other applications of the proposed methodology in statistical machine learning, reinforcement learning, and other scientific and engineering fields. Finally, it is also interesting to extend our methodology to the case of mixed constraints (i.e., equality and inequality constraint), and to develop novel methodology and analysis for constrained zeroth-order optimization with both binary and real-valued decision variables. 

\bibliographystyle{plainnat}
\bibliography{zoconstrained}

\clearpage
\onecolumn

\begin{center}
\Large\textbf{Supplementary Materials for ``Stochastic Zeroth-order Functional Constrained Optimization: Oracle Complexity and Applications"}
\end{center}

\section{Basics of Hamiltonian Monte Carlo sampling}\label{sec:HMCbasics}
For the sake of completeness, we give a brief description of the Hamiltonian Monte Carlo sampling algorithm used in Section~\ref{sec:hmctuning}. The presentation below follows~\cite{neal2011mcmc} for the most part. Suppose the problem is to sample from the distribution $\pi(q):\mathbb{R}^d \to \mathbb{R}$ whose potential function is given by $f(q): \mathbb{R}^d \to \mathbb{R}$. First \textcolor{black}{consider} the Hamiltonian form, given by 
\begin{align}\label{eq:hamilton}
H(q,p) = f(q) + K(p) = f(q) + p^\top M^{-1} p,
\end{align}
where $M \in \mathbb{R}^{d \times d}$ is the 'mass matrix'. Following~\cite{neal2011mcmc}, we assume a diagonal parametrization for $M$, i.e., we have $M = \kappa I$. The Hamiltonian dynamics of the position vector $q$ and the momentum vector $p$ is determined by the equation given by
\begin{align}
\frac{dz}{dt} = J \nabla H(z),\qquad\text{where}\qquad J= \begin{pmatrix}
  0_{d\times d} & I_{d \times d}\\ 
  -I_{d\times d} & 0_{d \times d}
\end{pmatrix}
\end{align}
and $z\coloneqq(q,p) \in \mathbb{R}^{2d}$ and $\nabla H$ is the gradient of the Hamiltonian function in~\eqref{eq:hamilton}. The HMC sampling algorithm is based on performing $\tau$ leapfrog steps for discretizing the above equation. Here, a \emph{leapfrog (or symplectic integrator) step}, for a given step-size $\eta$, is given by 
\begin{align*}
p_{n+1/2} &= p_n -\frac{\eta}{2} \frac{dH}{dq} (q_n) \\
q_{n+1} & = q_n +  \frac{\eta}{\kappa} p_{n+1/2}\\
p_{n+1} & = p_{n+1/2} -\frac{\eta}{2}  \frac{dH}{dq} (q_{n+1}),
\end{align*}
where $n$ is the index of the number of steps. More details regarding HMC could also be found in~\cite{betancourt2017conceptual}.

\section{Proofs for Section~\ref{sec:prelim}}\label{sec:prelimp}
We start with the following well-known result on the stochastic zeroth-order gradient estimator in~\eqref{eq:gradest}. 
\begin{theorem}[\cite{nesterov2017random}]\label{thm1}
For a Gaussian random vector $u \sim N(0, I_n)$ we have 
\begin{align}
    \mathbb{E}[\| u \|^k] & \leqslant (n+k)^{k/2}
\end{align}
for any $k \geqslant 2$. Moreover, the following statements hold for any function $\psi$ whose gradient is Lipschitz continuous with constant $L$
\begin{itemize}[leftmargin=0.1in]
    \item The gradient of $\psi_{\nu}(x)\coloneqq \mathbb{E}_u[\psi(x+\nu u)]$ is Lipschitz continuous with constant $L_{\nu}$ such that $L_{\nu} \leqslant L$.
    \item For any $x \in \mathbb{R}^n$, we have
    \begin{align}
        |\psi_{\nu}(x) - \psi(x) | & \leqslant \frac{\nu^2}{2}Ln,\\
        \| \nabla \psi_{\nu}(x) - \nabla \psi(x) \| & \leqslant \frac{\nu}{2}L(n+3)^{3/2}.
    \end{align}
    \item  For any $x \in \mathbb{R}^n$, we have
    \begin{align}
        \frac{1}{\nu^2}\mathbb{E}_u[\{\psi(x+\nu u) - \psi(x)\}^2\|u\|^2] \leqslant \frac{\nu^2}{2}L^2(n+6)^3 + 2(n+4)\| \nabla \psi(x)\|^2.
    \end{align}
\end{itemize}
\end{theorem}

\begin{proof}[\textbf{Proof of Lemma~\ref{lemma:zogradbounds}}]
Note that
\begin{align*}
    \| F_{\nu}(x, \xi, u) - f_{\nu}(x)\|^2 & = \textstyle\sum_{i=1}^m(f_{i, \nu_i}(x) - F_i(x+\nu_i u, \xi))^2.
\end{align*}
By Young's inequality, we have
\begin{align*}
    | F_i(x+\nu_iu, \xi) - f_{i, \nu_i}(x) |^2 & = | [F_i(x+\nu_iu, \xi) - F_i(x, \xi)] + [F_i(x, \xi) - f_i(x)] + [f_i(x) - f_{i, \nu_i}(x)] |^2 \notag \\
    & \leqslant 4| F_i(x+\nu_iu, \xi) - F_i(x, \xi)|^2 + 4|f_i(x) - f_{i, \nu_i}(x)|^2 + 2|F_i(x, \xi) - f_i(x)|^2 \notag \\
    & \leqslant 4M_{i}^2\nu_i^2\| u \|^2 + 4 \left(\frac{\nu_i^2}{2}L_in\right)^2 + 2| F_i(x, \xi) - f_i(x)|^2. \notag
\end{align*}
Now, by Assumption \ref{Assumption2} and Theorem \ref{thm1}, we have 
\begin{align*}
    \mathbb{E}|f_{i, \nu_i}(x) - F_i(x+\nu_iu, \xi)|^2 & \leq 4M_{i}^2\nu_i^2(n+2) + 2\sigma_{f, i}^2 + L_i^2\nu_i^4n^2. 
\end{align*}
Consequently, we obtain
\begin{align*}
    \mathbb{E}\| F_{\nu}(x, \xi, u) - f_{\nu}(x) \|^2 & \leqslant (\textstyle\sum_{i=1}^m4M_{i}^2\nu_i^2(n+2) + L_i^2\nu_i^4n^2) + 2\sigma_f^2 =: \sigma_{f, \nu}^2.
\end{align*}
\end{proof}

\begin{proof}[\textbf{Proof of Lemma~\ref{thm:gradestrate}}]
First note that by Theorem \ref{thm1}, we have
\begin{align}\label{temp1}
    &~\frac{1}{\nu_i^2}\mathbb{E}_u[\{ F_i(x+\nu_i u, \xi) - F_i(x, \xi) \}^2\|u\|^2] \notag \\\leqslant&~ \frac{\nu_i^2}{2}L_i^2(n+6)^3 + 2(n+4)\|\nabla F_i(x, \xi)\|^2 \notag \\  \leqslant &~\frac{\nu_i^2}{2}L_i^2(n+6)^3 + 4(n+4)[\| \nabla F_i(x, \xi) - \nabla f_i(x)\|^2 + \|\nabla f_i(x)\|^2].
\end{align}
Next note that
\begin{align*}
\| \nabla f_{i, \nu_i}(x) \| &\leqslant \| \nabla f_{i, \nu_i}(x) - \nabla f_i(x)\| + \| \nabla f_i(x)\| \\
&\leqslant \frac{\nu_i}{2}L_i(n+3)^{3/2} + L_iD_X + \| \nabla f_i(x^*)\| \\
&\leqslant \frac{\nu_i}{2}L_i(n+3)^{3/2} + L_iD_X + M_{i} =: \tilde{B}_i,
\end{align*}
where $M_{i}$ is from Assumption \ref{Assumption2}.
Taking the expectation with respect to $\xi$ on both sides of~\eqref{temp1}, we have 
\begin{align*}
    \mathbb{E}[|| G_{i, \nu_i}(x, \xi, u) \|^2] \leqslant \frac{\nu_i^2}{2}L_i^2(n+6)^3 + 4(n+4)[\sigma_i^2 + \tilde{B}_i^2].
\end{align*}

From the above inequalities, using Assumptions \ref{Assumption3} and \ref{Assumption2}, Theorem \ref{thm1}, and Young's inequality, we have
\begin{align*}
    \mathbb{E}[\| G_{i, \nu_i}(x, \xi, u) - \nabla f_{i, \nu_i}(x)\|^2] & \leqslant 2\mathbb{E}[\| G_{i, \nu_i}(x, \xi, u) \|^2] + 2\| \nabla f_{i, \nu_i}(x)\|^2 \\
    & \leqslant \nu_i^2 L_i^2(n+6)^3 + 8(n+4)[\sigma_i^2 + \tilde{B}_i^2] + 2\tilde{B}_i^2 \\ & \leqslant \nu_i^2L_i^2(n+6)^3 + 10(n+4)[\sigma_i^2 + \tilde{B}_i^2],
\end{align*}
which completes the proof. 
\end{proof}

\section{Proofs for Section~\ref{sec:thms}}\label{sec:thmsp}

In the proofs below, to avoid notational clutter, we use $x_t$ instead of using $x^{(t)}$, and we use $G_{i,\nu_i}(x_t, \xi_t, u_t)$ instead of $G_{i,\nu_i}(x^{(t)},\xi_i^{(t)}, u_i^{(t)})$. Next, in order to obtain the oracle complexity of Algorithm~\ref{Algorithm1}, we define a primal-dual gap function for the equivalent saddle point problem \eqref{eq:1.5}. In particular, given a pair of feasible solution $z = (x, y)$ and $\bar{z} = (\bar{x}, \bar{y})$ of \eqref{eq:1.5}, we define the primal-dual gap function $Q(z, \bar{z})$ as 
\begin{align}
    Q(z, \overline{z}) := \mathcal{L}(x, \bar{y}) - \mathcal{L}(\bar{x}, y). \label{eq:1.14}
\end{align}
For the remainder of the paper, we denote $Q_{\nu}(z, \bar{z}) = \mathcal{L}_{\nu}(x, \bar{y}) - \mathcal{L}_{\nu}(\bar{x}, y)$. Now we establish the error between these two functions.
\begin{lemma} \label{lemma4}
Under Assumptions~\ref{Assumption1},~\ref{Assumption3} and ~\ref{Assumption2}, we have
\begin{align}
    |Q(z, \bar{z}) - Q_{\nu}(z, \bar{z})| \leqslant \nu_0^2L_0n + M_Xn(\textstyle\sum_{i=1}^m\nu_i^4L_i^2)^{1/2},
\end{align}
where $M_X = \sup_{x \in X}\|x\|$.
\end{lemma}
\begin{proof}[\textbf{Proof of Lemma~\ref{lemma4}}]
First, we claim that the following is true:
\begin{align}
    \| f(x) - f_{\nu}(x)\| = \frac{n}{2}(\textstyle\sum_{i=1}^m\nu_i^4L_i^2)^{1/2}.
\end{align}
To see that, note that since the components $f_i$ of $f$ have continuous Lipschitz gradient and using theorem \ref{thm1}, we have 
\begin{align*}
    \| f(x) - f_{\nu}(x) \| & = (\textstyle\sum_{i=1}^m(f_i(x) - f_{i, \nu_i}(x))^2)^{1/2} \\
    & \leqslant \left(\sum_{i=1}^m\left(\frac{\nu_i^2L_in}{2}\right)^2\right)^{1/2} \\
    & = \left(\sum_{i=1}^m\frac{\nu_i^4}{4}L_i^2n^2\right)^{1/2} \\
    & = \frac{n}{2}(\textstyle\sum_{i=1}^m\nu_i^4L_i^2)^{1/2}
\end{align*}
Utilizing this relation, using Theorem \ref{thm1} and Cauchy-Schwartz inequality, we have
\begin{align*}
    | Q(z, \bar{z}) - Q_v(z, \bar{z})| & = | \mathcal{L}(x, \bar{y}) - \mathcal{L}(\bar{x}, y) - \mathcal{L}_{\nu}(x, \bar{y}) + \mathcal{L}_{\nu}(\bar{x}, y)| \\ 
    & = |f_0(x) + \bar{y}^Tf(x) - f_0(\bar{x}) - y^Tf(\bar{x}) - f_{0, \nu_0}(x) - \bar{y}^Tf_{\nu}(x) + f_{0, \nu_0}(\bar{x}) + y^Tf_{\nu}(\bar{x})| \\ 
    & \leqslant |f_0(x) - f_{0, \nu_0}(x)| + |f_0(\bar{x}) - f_{0, \nu_0}(\bar{x})| + |\bar{y}^T[f(x) - f_{\nu}(x)]| + |y^T[f(\bar{x}) - f_{\nu}(\bar{x})]| \\
    & \leqslant |f_0(x) - f_{0, \nu_0}(x)| + |f_0(\bar{x}) - f_{0, \nu_0}(\bar{x})| + \|\bar{y}\| \| f(x)-f_{\nu}(x)\| + \|y\| \|f(\bar{x}) - f_{\nu}(\bar{x})\| \\
    & \leqslant |f_0(x) - f_{0, \nu_0}(x)| + | f_0(\bar{x}) - f_{0, \nu_0}(\bar{x})| + M_X[\| f(x) - f_{\nu}(x)\| + \|f(\bar{x}) - f_{\nu}(\bar{x})\|] \\ & \leqslant \nu_0^2L_0n + M_X[n(\textstyle\sum_{i=1}^m\nu_i^4L_i^2)^{1/2}]
\end{align*}
This concludes the proof.
\end{proof}

\begin{lemma}\label{lemma222}
Suppose Assumptions~\ref{Assumption1},~\ref{Assumption3} and ~\ref{Assumption2} are satisfied. Then, for all $T \geqslant 1$, we have 
\begin{align}
    \mathbb{E}[f_0(\bar{x}_T) - f_0(x^*)] & \leqslant \frac{1}{\Gamma_T}\biggl[\gamma_0\eta_0W(x^*, x_0) + \frac{\gamma_0\eta_0}{2}\| y_0\|_2^2 + \textstyle\sum_{t=0}^{T-1}\frac{2\gamma_t}{\eta_t - L_{0} - L_{f}}\mathbb{E}[\| \delta_t^G\|_*^2] \notag \\   &\quad+  \left(\textstyle\sum_{t=1}^{T-1}\frac{12\gamma_t\theta_t^2}{\tau_t} + \frac{12\gamma_{T-1}}{\tau_{T-1}}\right)(\sigma_{f, \nu}^2 + D_X^2\| \sigma_{\nu}\|_2^2)\biggr] + [\nu_0^2L_0n + M_Xn(\textstyle\sum_{i=1}^m\nu_i^4L_i^2)^{1/2}] \label{eq:1.37} \\
    \mathbb{E}[\| [f(\bar{x}_T)]_+ \|_2] & \leqslant \frac{1}{\Gamma_T}\biggl[\gamma_0\tau_0\|y_0\|_2^2 + 3(\|y^*\|_2 + 1)^2\gamma_0\tau_0 + \gamma_0\eta_0W(x^*, x_0) \notag \\ &\quad + \textstyle\sum_{t=0}^{T-1}\frac{2\gamma_t}{\eta_t - L_{0} - L_{f}}\left\{\mathbb{E}[\|\delta_t^G\|_*^2] + \left(\frac{L_{f}D_X}{2}\|y^*\|_2\right)^2\right\} \label{eq:1.49} \\ &\quad + \left(\textstyle\sum_{t=1}^{T-1}\frac{12\gamma_t\theta_t^2}{\tau_t} + \sum_{t=0}^{T-1}\frac{\gamma_t}{\tau_t} + \frac{12\gamma_{T-1}}{\tau_{T-1}}\right)(\sigma_{f, \nu}^2 + D_X^2\|\sigma_{\nu}\|_2^2)\biggr] \notag \\ &\quad+ [\nu_0^2L_0n + M_Xn(\textstyle\sum_{i=1}^m\nu_i^4L_i^2)^{1/2}]. \notag
\end{align}
where $\Gamma_T := \sum_{t=0}^{T-1}\gamma_t$ and $\sigma_{\nu} = (\sigma_{1, \nu_1}, \ldots, \sigma_{m, \nu_m})$ with $\sigma_{i, \nu_i}$ as defined in \eqref{eq:1.30}, and \textcolor{black}{$\delta^G_t\coloneqq  G_{0, \nu_0}(x_t, \xi_t, u_t) - f_{0, \nu_0}'(x_t) + \sum_{i=1}^my_{t+1}^{(i)} \left(G_{i, \nu_i}(x_t, \xi_t, u_t) - f_{i, \nu_i}'(x_t) \right) $}.
\end{lemma}

\begin{proof}[\textbf{Proof of Lemma~\ref{lemma222}}]
First, observe that $y_{t+1}$ is a constant conditioned on random variable $\xi_{[t-1]}, u_{[t-1]}, \bar{\xi}_{[t-1]}, \bar{u}_{[t-1]}$. In particular,
\begin{align}
    \mathbb{E}[\langle \delta_t^G, x_t - x\rangle] = \mathbb{E}\langle \mathbb{E}_{| \xi_{[t-1]}, u_{[t-1]}, \bar{\xi}_{[t-1]}, \bar{u}_{[t-1]}}[\delta_t^G], x_t - x \rangle = 0 \label{eq:1.31}
\end{align}
for any non-random $x$. This follows due to the following relation
\begin{align*}
    &~~\mathbb{E}_{| \xi_{[t-1]}, u_{[t-1]}, \bar{\xi}_{[t-1]}, \bar{u}_{[t-1]}}[\delta_t^G] \\
    =& ~~\mathbb{E}_{| \xi_{[t-1]}, u_{[t-1]}, \bar{\xi}_{[t-1]}, \bar{u}_{[t-1]}}[G_{0, \nu_0}(x_t, \xi_t, u_t) - f_{0, \nu_0}'(x_t)] \\
    &~~~+ \textstyle\sum_{i=1}^my_{t+1}^{(i)}\mathbb{E}_{| \xi_{[t-1]}, u_{[t-1]}, \bar{\xi}_{[t-1]}, \bar{u}_{[t-1]}}[G_{i, \nu_i}(x_t, \xi_t, u_t) - f_{i, \nu_i}'(x_t)] \\
    = &~~ \mathbf{0}.
\end{align*}
Similarly, we have 
\begin{align}
    \mathbb{E}[\langle \delta_{t+1}^F, y_{t+1} - y\rangle]= \mathbb{E}[\langle \mathbb{E}_{| \xi_{[t]}, u_{[t]}, \bar{\xi}_{[t-1]}, \bar{u}_{[t-1]}}[\delta_{t+1}^F], y_{t+1} - y\rangle] = 0, \label{eq:1.32}
\end{align}
for any non-random $y$. Here, we note that 
\begin{align}
    \mathbb{E}_{| \xi_{[t]}, u_{[t]}, \bar{\xi}_{[t-1]}, \bar{u}_{[t-1]}}[\delta_{t+1}^F] & = \mathbb{E}_{| \xi_{[t]}, u_{[t]}, \bar{\xi}_{[t-1]}, \bar{u}_{[t-1]}}[F_{\nu}(x_t, \bar{\xi}_t, \overline{u}_t)] - f_{\nu}(x_t) \notag \\ & + ( \mathbb{E}_{| \xi_{[t]}, u_{[t]}, \bar{\xi}_{[t-1]}, \bar{u}_{[t-1]}}[\mathbf{G}_{\nu}(x_t, \bar{\xi}_t, \bar{u}_t)] - f_{\nu}'(x_t))^T(x_{t+1} - x_t) = \mathbf{0}, \label{eq:1.41}
\end{align}
where the first term in RHS is $\mathbf{0}$ due to $\mathbb{E}_{\xi, u}F_{\nu}(x, \xi, u) = f_{\nu}(x)$, the second term is $\mathbf{0}$ due to the $\mathbb{E}_{\xi, u}\mathbf{G}_{\nu}(x, \xi, u) = f_{\nu}'(x)$ and the common fact for both the terms that $x_t, x_{t+1}$ are constants for given $\xi_{[t]}, u_{[t]}, \bar{\xi}_{[t-1]}, \bar{u}_{[t-1]}$. We now note that
\begin{align}
   \hspace{-0.2in} \mathbb{E}[\| \delta_t^F\|_2^2] & \leqslant 2\mathbb{E}[\| F_{\nu}(x_{t-1}, \bar{\xi}_{t-1}, \bar{u}_{t-1}) - f_{\nu}(x_{t-1})\|_2^2] + 2\mathbb{E}[\| [\mathbf{G}_{\nu}(x_{t-1}, \bar{\xi}_{t-1}, \bar{u}_{t-1}) - f_{\nu}'(x_{t-1})]^T(x_t - x_{t-1}) \|_2^2] \notag \\ 
    & \leqslant 2\sigma_{f, \nu}^2 + 2\mathbb{E}\left[\sum_{i=1}^m\left\{(G_{i, \nu_i}(x_{t-1}, \bar{\xi}_{t-1}, \bar{u}_{t-1}) - f_{i, \nu_i}'(x_{t-1}))^T(x_t - x_{t-1})\right\}^2\right] \notag \\
    & \leqslant 2\sigma_{f, \nu}^2 + 2\mathbb{E}\left[\sum_{i=1}^m\| G_{i, \nu_i}(x_{t-1}, \bar{\xi}_{t-1}, \bar{u}_{t-1}) - f_{i, \nu_i}'(x_{t-1})\|_*^2 \|x_t - x_{t-1}\|^2\right] \notag \\
    & \leqslant 2\sigma_{f, \nu}^2 + 2D_X^2\|\sigma_{\nu}\|_2^2. \label{eq:1.44}
\end{align}
Then, in view of above relation and definitions of $q_t, \bar{q}_t$, and by defining $\delta_t^F\coloneqq  \ell_F(x_t) - \ell_f(x_t)$, we have 
\begin{align}
    \mathbb{E}[\| q_t - \bar{q}_t\|_2^2] & = \mathbb{E}[\| \ell_F(x_t) - \ell_f(x_t) - \ell_F(x_{t-1}) + \ell_f(x_{t-1}) \|_2^2] \notag \\ 
    & \leqslant 2\mathbb{E}[\| \delta_t^F\|_2^2] + 2\mathbb{E}[\| \delta_{t-1}^F\|_2^2] \leqslant 8(\sigma_{f, \nu}^2 + D_X^2\| \sigma_{\nu} \|_2^2). \label{eq:1.33}
\end{align}
Taking the expectation on both sides of \eqref{eq:1.29} and using relation \eqref{eq:1.31}, \eqref{eq:1.32} and \eqref{eq:1.33}, we have for all non-random $z \in \{ (x, y) : x \in X, y \geqslant \mathbf{0} \}$,
\begin{align}
    & \mathbb{E}\left[\textstyle\sum_{t=0}^{T-1}\gamma_tQ_{\nu}(z_{t+1}, z)\right]  \notag \\
     \leqslant & \gamma_0\eta_0W(x,x_0) - \gamma_{T-1}\eta_{T-1}\mathbb{E}[W(x, x_T)] + \frac{\gamma_0\tau_0}{2}\| y - y_0\|_2^2 \notag \\
     + & \textstyle\sum_{t=0}^{T-1}\frac{2\gamma_t}{\eta_t - L_{0} - L_{f}}\left[\mathbb{E}[\| \delta_t^G\|_*^2] + \left(\frac{L_{f}D_X}{2}[\| y \|_2 - 1]_+\right)^2\right] \notag \\
     + & \left(\textstyle\sum_{t=1}^{T-1}\frac{12\gamma_t\theta_t^2}{\tau_t} + \frac{12\gamma_{T-1}}{\tau_{T-1}}\right)(\sigma_{f, \nu}^2 + D_X^2\| \sigma_{\nu}\|_2^2) \label{eq:1.34}
\end{align}
where we dropped $\| y - y_T \|_2^2$. By Lemma \ref{lemma4}, we have 
\begin{align*}
    Q(z_{t+1}, z) - [\nu_0^2L_0n + M_Xn(\textstyle\sum_{i=1}^m\nu_i^4L_i^2)^{1/2}] \leqslant Q_{\nu}(z_{t+1}, z).
\end{align*}

Using this relation, multiplying both sides by $\gamma_t$, summing from $t = 0, \dots, T-1$, and taking expectation on both sides, we have 
\begin{align}
    \mathbb{E}\left[\textstyle\sum_{t=0}^{T-1}\gamma_tQ(z_{t+1}, z)\right] - [\nu_0^2L_0n + M_Xn(\textstyle\sum_{i=1}^m\nu_i^4L_i^2)^{1/2}]\Gamma_T \leqslant \mathbb{E}\left[\sum_{t=0}^{T-1}\gamma_tQ_{\nu}(z_{t+1}, z)\right] \label{eq:1.35}
\end{align}
Using this relation, the convexity of $f_0(\cdot)$ and $f(\cdot)$, and noting the definition of $\Gamma_T$, we have for all non-random $y \geqslant \mathbf{0}$ and $x \in X$,
\begin{align}
    &~~ \Gamma_T\mathbb{E}[f_0(\bar{x}_T) + \langle y, f(\bar{x}_T)\rangle - f_0(x) - \langle \bar{y}_T, f(x)\rangle] - [\nu_0^2L_0n + M_Xn(\textstyle\sum_{i=1}^m\nu_i^4L_i^2)^{1/2}]\Gamma_T \notag \\
     \leqslant &~~\mathbb{E}\left[\textstyle\sum_{t=0}^{T-1}\gamma_tQ(z_{t+1}, z)\right] - [\nu_0^2L_0n + M_Xn(\textstyle\sum_{i=1}^m\nu_i^4L_i^2)^{1/2}] \Gamma_T \notag \\
    \leqslant&~~ \mathbb{E}\left[\textstyle\sum_{t=0}^{T-1}\gamma_tQ_{\nu}(z_{t+1}, z)\right]. \label{eq:1.36}
\end{align}
Combining \eqref{eq:1.34}, \eqref{eq:1.35} and \eqref{eq:1.36}, then choosing $x = x^*$, $y = \mathbf{0}$ (which are non-random) throughout the combined relation, observing that $[0 - 1]_+ = 0$, we have 
\begin{equation}
\begin{aligned}[b]
    &~~ \Gamma_T\mathbb{E}[f_0(\bar{x}_T) - f_0(x^*) - \langle \bar{y}_T, f(x^*)\rangle] - [\nu_0^2L_0n + M_Xn(\textstyle\sum_{i=1}^m\nu_i^4L_i^2)^{1/2}]\Gamma_T \\ \leqslant &~~ \mathbb{E}\left[\textstyle\sum_{t=0}^{T-1}\gamma_tQ_{\nu}(z_{t+1}, (x^*, \mathbf{0}))\right] \\
     \leqslant &~~\gamma_0\eta_0W(x^*, x_0) - \gamma_{T-1}\eta_{T-1}\mathbb{E}[W(x^*, x_T)] + \frac{\gamma_0\tau_0}{2}\|y_0\|_2^2 + \textstyle\sum_{t=0}^{T-1}\frac{2\gamma_t}{\eta_t - L_{0} - L_{f}}\mathbb{E}[\| \delta_t^G\|_*^2] \\  + &~~  \left(\textstyle\sum_{t=1}^{T-1}\frac{12\gamma_t\theta_t^2}{\tau_t} + \frac{12\gamma_{T-1}}{\tau_{T-1}}\right)(\sigma_{f, \nu}^2 + D_X^2\| \sigma_{\nu} \|_2^2)
\end{aligned}
\end{equation}
Ignoring the $\mathbb{E}[W(x^*, x_T)]$ term and noting that $f(x^*) \leqslant \mathbf{0}$ and $\bar{y}_T \geqslant \mathbf{0}$ implies $\langle \bar{y}_T, f(x^*)\rangle \leqslant 0$, we have \eqref{eq:1.37}. \\

\quad Now, we focus our attention to the infeasibility bound. First, we define $R := \|y^*\|_2 + 1$. Second, define an auxilliary sequence $\{y_t^v\}$ in the following way: $y_0^v = y_0$ and for all $t \geqslant 0$, define 
\begin{align*}
    y_{t+1}^v := \arg\min_{y \in \mathcal{B}_+^2(R)}\frac{1}{\tau_{t-1}}\langle \delta_t^F, y\rangle + \frac{1}{2}\| y - y_t^v\|_2^2, 
\end{align*}
where we recall that $\mathcal{B}_+^2(R) = \{ x \in \mathbb{R}^n : \| x \|_2 \leqslant R, x \geqslant \mathbf{0}\}$. Then in view of Lemma \ref{lemma7}, in particular relation \eqref{eq:1.38}, for all $y \in \mathcal{B}_+^2(R)$ we have 
\begin{align}
    \frac{1}{\tau_t}\langle \delta_{t+1}^F, y_{t+1}^v - y\rangle \leqslant \frac{1}{2}\| y - y_{t+1}^v\|_2^2 - \frac{1}{2}\|y - y_{t+2}^v\|_2^2 + \frac{1}{2\tau_t^2}\| \delta_{t+1}^F\|_2^2. \label{eq:1.39}
\end{align}
Multiplying \eqref{eq:1.39} by $\gamma_t\tau_t$, taking a sum from $t = 0$ to $T - 1$ and noting the second relation in \eqref{eq:1.22}, we obtain
\begin{align}
    \textstyle\sum_{t=0}^{T-1}\gamma_t\langle \delta_{t+1}^F, y_{t+1}^v - y\rangle \leqslant \frac{\gamma_0\tau_0}{2}\|y - y_1^v\|_2^2 + \sum_{t=0}^{T-1}\frac{\gamma_t}{2\tau_t}\| \delta_{t+1}^F\|_2^2, \label{eq:1.40}
\end{align}
for all $y \in \mathcal{B}_+^2(R)$. Summing \eqref{eq:1.40} and \eqref{eq:1.29}, we obtain
\begin{align}
    & \textstyle\sum_{t=0}^{T-1}\gamma_tQ_{\nu}(z_{t+1}, z) + \sum_{t=0}^{T-1}\gamma_t[\langle \delta_t^G, x_t - x\rangle - \langle \delta_{t+1}^F, y_{t+1} - y_{t+1}^{v}\rangle] \notag \\
    & \leqslant \frac{\gamma_0\tau_0}{2}[\| y - y_0\|_2^2 + \|y-y_1^v\|_2^2] + \gamma_0\eta_0W(x, x_0) \notag \\
    & + \textstyle\sum_{t=1}^{T-1}\frac{3\gamma_t\theta_t^2}{2\tau_t}\| q_t - \bar{q}_t\|_2^2 + \frac{3\gamma_{T-1}}{2\tau_{T-1}}\| q_T - \bar{q}_T\|_2^2 \notag \\
    & + \textstyle\sum_{t=0}^{T-1}\left[\frac{2\gamma_t}{\eta_t - L_{0} - L_{f}}\left\{\| \delta_t^G\|_*^2 + \left(\frac{L_{f}D_X}{2}[\| y \|_2 - 1]_+\right)^2\right\} + \frac{\gamma_t}{2\tau_t}\| \delta_{t+1}^F\|_2^2\right], \label{eq:1.42}
\end{align}
for all $z \in \{(x, y) : x \in X, y \in \mathcal{B}_+^2(R)\}$. Note that given $\xi_{[t]}, u_{[t]}$ and $\bar{\xi}_{[t-1]}, \bar{u}_{[t-1]}$, we have $y_{t+1}, y_{t+1}^v, x_{t+1}, x_t$ are constants. Hence, we have 
\begin{align}
    \mathbb{E}[\langle \delta_{t+1}^F, y_{t+1} - y_{t+1}^v\rangle] = \mathbb{E}[\langle \mathbb{E}_{| \xi_{[t]}, u_{[t]}, \bar{\xi}_{[t-1]}, \bar{u}_{[t-1]}}[\delta_{t+1}^F] , y_{t+1} - y_{t+1}^v \rangle ] = 0, \label{eq:1.43}
\end{align}
where second equality follows from \eqref{eq:1.41}. Choosing $z = \widehat{z} := (x^*, \widehat{y})$ in \eqref{eq:1.42} where $\widehat{y} := (\|y^*\|_2+1)[f(\bar{x}_T)]_+\| [f(\bar{x}_T)]_+ \|_2^{-1} \in \mathcal{B}_+^2(R)$, taking expectation on both sides and noting \eqref{eq:1.43}, \eqref{eq:1.44}, \eqref{eq:1.33}, first relation in \eqref{eq:1.31}, we have
\begin{align}
    \mathbb{E}\left[\textstyle\sum_{t=0}^{T-1}\gamma_tQ_{\nu}(z_{t+1}, \hat{z})\right] & \leqslant \frac{\gamma_0\tau_0}{2}\mathbb{E}[\| \hat{y} - y_0\|_2^2 + \| \hat{y} - y_1^v \|_2^2] + \gamma_0\eta_0W(x^*, x_0) \notag \\
    & + \textstyle\sum_{t=0}^{T-1}\frac{2\gamma_t}{\eta_t - L_{0} - L_{f}}\left\{\mathbb{E}[\| \delta_t^G\|_*^2] +  \left(\frac{L_{f}D_X}{2}\| y^*\|_2\right)^2\right\} \notag \\
    & + \left(\textstyle\sum_{t=1}^{T-1}\frac{12\gamma_t\theta_t^2}{\tau_t} + \textstyle\sum_{t=0}^{T-1}\frac{\gamma_t}{\tau_t} +  \frac{12\gamma_{T-1}}{\tau_{T-1}}\right)(\sigma_{f, \nu}^2 + D_X^2\| \sigma_{\nu}\|_2^2). \label{eq:1.46}
\end{align}
By Lemma \ref{lemma4}, we then have $Q(z_{t+1}, \hat{z}) - [\nu_0^2L_0n + M_Xn(\textstyle\sum_{i=1}^m\nu_i^4L_i^2)^{1/2}] \leqslant Q_{\nu}(z_{t+1}, \hat{z})$. Multiplying both sides by $\gamma_t$, summing from $t = 0$ to $T-1$, taking expectation of both sides and dividing by $\Gamma_T$, we have 
\begin{align}
    \frac{1}{\Gamma_T}\mathbb{E}\left[\textstyle\sum_{t=0}^{T-1}\gamma_tQ(z_{t+1}, \hat{z})\right] - [\nu_0^2L_0n + M_Xn(\textstyle\sum_{i=1}^m\nu_i^4L_i^2)^{1/2}] \leqslant \frac{1}{\Gamma_T}\mathbb{E}\left[\sum_{t=0}^{T-1}\gamma_tQ_{\nu}(z_{t+1}, \hat{z})\right] \label{eq:1.48}
\end{align}
Noting the convexity of $Q$ in the first argument, we obtain
\begin{align}
    \mathbb{E}[Q(\bar{z}_T, \hat{z})] \leqslant \frac{1}{\Gamma_T}\mathbb{E}\left[\textstyle\sum_{t=0}^{T-1}\gamma_tQ(z_{t+1}, \hat{z})\right]. \label{eq:1.47}
\end{align}
Now observe that we have $\mathcal{L}(\bar{x}_T, y^*) - \mathcal{L}(x^*, y^*) \geqslant 0$ which implies that $f_0(\bar{x}_T) + \langle y^*, f(\bar{x}_T)\rangle - f_0(x^*) \geqslant 0$, which follows from complementary slackness. In view of the relation
\begin{align*}
    \langle y^*, f(\bar{x}_T)\rangle \leqslant \langle y^*, [f(\bar{x}_T)]_+\rangle \leqslant \| y^* \|_2 \| [f(\bar{x}_T)]_+\|_2,
\end{align*}
the above inequality  implies that 
\begin{align}
    f_0(\bar{x}_T) + \| y^*\|_2 \| [f(\bar{x}_T)]_+\|_2 - f_0(x^*) \geqslant 0. \label{eq:1.45}
\end{align}
Moreover, we have that
\begin{align*}
    Q(\bar{z}_T, \hat{z}) = \mathcal{L}(\bar{x}_T, \hat{y}) - \mathcal{L}(x^*, \bar{y}_T) \geqslant \mathcal{L}(\bar{x}_T, \hat{y}) - \mathcal{L}(x^*, y^*) = f_0(\bar{x}_T) + (\| y^* \|_2 + 1)\| [ f(\bar{x}_T)]_+\|_2 - f_0(x^*),
\end{align*}
which along with \eqref{eq:1.45} implies that
\begin{align*}
    Q(\bar{z}_T, \hat{z}) \geqslant \|[f(\bar{x}_T)]_+\|_2.
\end{align*}
The above relation, \eqref{eq:1.46}, \eqref{eq:1.48} and \eqref{eq:1.47} together yield 
\begin{align*}
    \mathbb{E}[\| [f(\bar{x}_T)]_+ \|_2] & \leqslant \frac{1}{\Gamma_T}\biggl[\frac{\gamma_0\tau_0}{2}\mathbb{E}[\|\hat{y} - y_0\|_2^2 + \| \hat{y} - y_1^v \|_2^2] + \gamma_0\eta_0 W(x^*, x_0) \\ & \quad + \textstyle\sum_{t=0}^{T-1}\frac{2\gamma_t}{\eta_t - L_{0} - L_{f}}\left\{\mathbb{E}[\| \delta_t^G\|_*^2] + \left(\frac{L_{f}D_X}{2}\|y^*\|_2\right)^2\right\} \\ 
    & \quad + \left(\textstyle\sum_{t=1}^{T-1}\frac{12\gamma_t\theta_t^2}{\tau_t} + \sum_{t=0}^{T-1}\frac{\gamma_t}{\tau_t} + \frac{12\gamma_{T-1}}{\tau_{T-1}}\right)(\sigma_{f, \nu}^2 + D_X^2\|\sigma_{\nu}\|_2^2)\biggr] \\ 
    & \quad+ [\nu_0^2L_0n + M_Xn(\textstyle\sum_{i=1}^m\nu_i^4L_i^2)^{1/2}].
\end{align*}
Noting the bound $\| \hat{y} - y_1^v\| \leqslant 2R$ and $\| \hat{y} - y_0\|_2^2 \leqslant 2\|y_0\|_2^2 + 2\|\hat{y}\|_2^2 \leqslant 2\|y_0\|_2^2 + 2R^2$ in the above relation and recalling that $R = \|y^*\|_2 + 1$, we obtain \eqref{eq:1.49}. Hence, we conclude the proof. 
\end{proof}
We next bound the term  $\mathbb{E}[\| \delta_t^G \|_*^2]$ appearing in the previous result in the zeroth-order setting. This result is crucial for obtaining a linear dependency on the number of constraints $m$ for our oracle complexity results and is based on our Lemma \ref{lemma4}.
\begin{lemma} \label{lemma10}
Assume that $\{ \gamma_t, \tau_t, \eta_t \}$ satisfy
\begin{align}
    \frac{96\| \sigma_{\nu} \|_2^2}{\tau_t(\eta_t - L_{0} - L_{f})} < 1, \label{eq:1.52}
\end{align}
for all $t \leqslant T-1$ and constants $R_1$ and $R_2$ satisfying the following conditions exist:
\begin{equation}
\begin{aligned}[b]
    R_1 & \geqslant \left(1 - \frac{96 \| \sigma_{\nu}\|_2^2}{\tau_t(\eta_t - L_{0} - L_{f})}\right)^{-1}\biggl[2\sigma_{0, \nu_0}^2  + \frac{48\| \sigma_{\nu}\|_2^2}{\gamma_t\tau_t}\biggl\{\gamma_0\eta_0W(x^*, x_0) + \frac{\gamma_0\tau_0}{2}\| y^* - y_0\|_2^2 + \frac{\gamma_t\tau_t}{12}\|y^*\|_2^2 \\ & + \textstyle\sum_{i=0}^t\frac{2\gamma_i}{\eta_i - L_{0} - L_{f}}\left(\frac{L_{f}D_X}{2}[\|y^*\|_2 - 1]_+\right)^2 + \left(\sum_{i=1}^t\frac{12\gamma_i\theta_i^2}{\tau_i} + \frac{12\gamma_t}{\tau_t}\right)(\sigma_{f, \nu}^2 + D_X^2\|\sigma_{\nu}\|_2^2) \\ & + [\nu_0^2L_0n + M_Xn(\textstyle\sum_{i=1}^m\nu_i^4L_i^2)^{1/2}]\Gamma_{t+1}\biggr\}\biggr] \label{eq:1.55}
\end{aligned}
\end{equation}
for all $t \leqslant T-1$ and 
\begin{align}
    R_2 \geqslant \left(1 - \frac{96\| \sigma_{\nu}\|_2^2}{\tau_t(\eta_t - L_{0} - L_{f})}\right)^{-1}\frac{96\| \sigma_{\nu}\|_2^2\gamma_i}{\gamma_t\tau_t(\eta_i - L_{0} - L_{f})} \label{eq:1.56}
\end{align}
for all $t \leqslant T-1$ and $i \leqslant t-1$. Then, we have 
\begin{align}
    \mathbb{E}[\| \delta_t^G \|_*^2] \leqslant R_1(1 + R_2)^t, \label{eq:1.53}
\end{align}
for all $t \leqslant T-1$. In particular, if $\| \sigma_{\nu}\|_2 = 0$, then we can set $R_1 = 2\sigma_{0, \nu_0}^2$ and $R_2 = 0$ implying $\mathbb{E}[\| \delta_t^G\|_*^2] \leqslant 2\sigma_{0, \nu_0}^2$.
\end{lemma}

\begin{proof}[\textbf{Proof of Lemma~\ref{lemma10}}]
First note that by Lemma \ref{lemma4}, we have 
\begin{align*}
    Q(z_{i+1}, z) - [\nu_0^2L_0n + M_Xn(\textstyle\sum_{i=1}^m\nu_i^4L_i^2)^{1/2}] \leqslant Q_{\nu}(z_{i+1}, z)
\end{align*}
Multiplying the above by $\gamma_i$ and summing up $i = 0$ to $t$, we have 
\begin{align*}
    \textstyle\sum_{i=0}^{t}\gamma_iQ(z_{i+1}, z) - [\nu_0^2L_0n + M_Xn(\textstyle\sum_{i=1}^m\nu_i^4L_i^2)^{1/2}]\Gamma_{t+1} \leqslant \sum_{i=0}^t\gamma_iQ_{\nu}(z_{i+1}, z)
\end{align*}
Replacing $T$ for $t+1 (\geqslant 1)$ in \eqref{eq:1.29}, we have 
\begin{align}
    & \textstyle\sum_{i=0}^{t}\gamma_iQ_{\nu}(z_{i+1}, z) + \sum_{i=0}^{t}\gamma_i[\langle \delta_i^G, x_i - x\rangle - \langle \delta_{i+1}^F, y_{i+1} - y\rangle] \notag \\
    & \leqslant \gamma_0 \eta_0 W(x, x_0) - \gamma_{t}\eta_{t}W(x, x_{t+1}) + \frac{\gamma_0\tau_0}{2}\| y - y_0 \|_2^2 - \frac{\gamma_{t}\tau_{t}}{12}\|y - y_{t+1}\|_2^2 \notag \\
    & + \textstyle\sum_{i=0}^{t}\frac{2\gamma_i}{\eta_i - L_{0} - L_{f}}\left[\| \delta_i^G\|_*^2 + \left(\frac{L_{f}D_X}{2}[\|y\|_2 - 1]_+\right)^2\right] \notag \\
    & + \textstyle\sum_{i=1}^{t}\frac{3\gamma_i\theta_i^2}{2\tau_i}\|q_i - \bar{q}_i\|_2^2 + \frac{3\gamma_{t}}{2\tau_{t}}\|q_{t+1} - \bar{q}_{t+1}\|_2^2.
\end{align}
Observe that $Q(z_{i+1}, z^*) \geqslant 0$ for $i = 0, \dots, t$ by our saddle point assumption where $z^* = (x^*, y^*)$. Choosing $z = z^*$ (both non-random) in the above relations, taking expectation, using \eqref{eq:1.31} with $x = x^*$ and \eqref{eq:1.32} with $y = y^*$, disregarding the term $-\gamma_t\eta_t\mathbb{E}[W(x^*, x_{t+1})]$ and noting \eqref{eq:1.33}, we have the following inequality
\begin{align}
 &~~   -[\nu_0^2L_0n + M_Xn(\textstyle\sum_{i=1}^m\nu_i^4L_i^2)^{1/2}]\Gamma_{t+1} + \frac{\gamma_t\tau_t}{12}\mathbb{E}\| y^* - y_{t+1}\|_2^2  \\
   & \leqslant \gamma_0\eta_0W(x^*, x_0) + \frac{\gamma_0\tau_0}{2}\| y^* - y_0\|^2 \notag \\ & + \textstyle\sum_{i=0}^t\frac{2\gamma_i}{\eta_i - L_{0} - L_{f}}\left[\mathbb{E}[\|\delta_i^G\|_*^2] + \left(\frac{L_{f}D_X}{2}[\|y^*\|_2 - 1]_+\right)^2\right] \notag \\ & + \left(\textstyle\sum_{i=1}^t\frac{12\gamma_i\theta_i^2}{\tau_i} + \frac{12\gamma_t}{\tau_t}\right)(\sigma_{f, \nu}^2 + D_X^2\| \sigma_{\nu}\|_2^2) \label{eq:1.50}
\end{align}
Now, let us define $\delta_{t, i}^G := G_{i, \nu_i}(x_t, \xi_t, u_t) - f_{i, \nu_i}'(x_t)$ for $i = 0, \dots, m$. As a consequence, we have $\delta_t^G = \delta_{t, 0}^G + \sum_{i=1}^my_{t+1}^{(i)}\delta_{t, i}^G$. Then, we have 
\begin{align}
    \mathbb{E}[\| \delta_t^G \|_*^2] & = \mathbb{E}[\| \delta_{t, 0}^G + \textstyle\sum_{i=1}^my_{t+1}^{(i)}\delta_{t, i}^G\|_*^2] \notag \\ & \overset{(i)}{\leqslant} 2\mathbb{E}[\| \delta_{t, 0}^G\|_*^2] + 2\mathbb{E}[\| \textstyle\sum_{i=1}^m y_{t+1}^{(i)}\delta_{t, i}^G \|_*^2] \notag \\ 
    & \leqslant 2\mathbb{E}[\| \delta_{t, 0}^G\|_*^2] + 2\mathbb{E}[(\textstyle\sum_{i=1}^m\| y_{t+1}^{(i)}\delta_{t, i}^G\|)^2] \notag \\
    & \overset{(ii)}{\leqslant} 2[\sigma_{0, \nu_0}^2 + \mathbb{E}[\| y_{t+1}\|_2^2(\textstyle\sum_{i=1}^m\| \delta_{t, i}^G\|_*^2)]] \notag \\
    & \overset{(iii)}{\leqslant} 2[\sigma_{0, \nu_0}^2 + \mathbb{E}[\|y_{t+1}\|_2^2(\textstyle\sum_{i=1}^m\mathbb{E}_{| \xi_{[t-1]}, u_{[t-1]}, \bar{\xi}_{[t-1]}, \bar{u}_{[t-1]}}[\| \delta_{t, i}^G\|_*^2])]] \notag \\
    & \overset{(iv)}{\leqslant} 2[\sigma_{0, \nu_0}^2 + \mathbb{E}[\| y_{t+1} \|_2^2\textstyle\sum_{i=1}^m\sigma_{i, \nu_i}^2]] \notag \\
    & = 2(\sigma_{0, \nu_0}^2 + \| \sigma_{\nu}\|_2^2 \mathbb{E}\| y_{t+1} \|_2^2) \notag \\
    & \leqslant 2\sigma_{0, \nu_0}^2 + 4\| \sigma_{\nu}\|_2^2(\| y^*\|_2^2 + \mathbb{E}[\| y_{t+1} - y^*\|_2^2]). \label{eq:1.51}
\end{align}
Here, relation (i) follows due to the fact that $\| a + b \|_*^2 \leqslant (\| a\|_* + \|b\|_*)^2 \leqslant 2\|a\|_*^2 + 2\|b\|_*^2$, relation (ii) follows due to Cauchy-Schwarz inequality, relation (iii) follows due to the fact that $y_{t+1}$ is a constant conditioned on random variables $\xi_{[t-1]}, u_{[t-1]}, \bar{\xi}_{[t-1]}, \bar{u}_{[t-1]}$ and relation (iv) follows from the fact that $x_t$ is a constant conditioned on random variables $\xi_{[t-1]}, u_{[t-1]}, \bar{\xi}_{[t-1]}, \bar{u}_{[t-1]}$. \\

Adding $\frac{\gamma_t\tau_t}{12}\|y^*\|_*^2$ to both sides of \eqref{eq:1.50}, then multiplying it by $\frac{48\| \sigma_{\nu}\|_2^2}{\gamma_t\tau_t}$ and observing \eqref{eq:1.51}, we have 
\begin{align*}
    \mathbb{E}[\| \delta_t^G\|_*^2] & \leqslant 2\sigma_{0, \nu_0}^2 + \frac{48\| \sigma_{\nu}\|_2^2}{\gamma_t\tau_t}\biggl\{\gamma_0\eta_0W(x^*, x_0) + \frac{\gamma_0\tau_0}{2}\| y^* - y_0\|_2^2 + \frac{\gamma_t\tau_t}{12}\|y^*\|_2^2 \\ 
    &~~~~ + \textstyle\sum_{i=0}^t\frac{2\gamma_i}{\eta_i - L_{0} - L_{f}}\left(\frac{L_{f}D_X}{2}[\|y^*\|_2 - 1]_+\right)^2 \\
    &~~~~ + \left(\textstyle\sum_{i=1}^t\frac{12\gamma_i\theta_i^2}{\tau_i} + \frac{12\gamma_t}{\tau_t}\right)(\sigma_{f, \nu}^2 + D_X^2\| \sigma_{\nu}\|_2^2) + [\nu_0^2L_0n + M_Xn(\textstyle\sum_{i=1}^m\nu_i^4L_i^2)^{1/2}]\Gamma_{t+1} \biggr\} \\ & ~~~~+ \textstyle\sum_{i=0}^t\frac{96\| \sigma_{\nu}\|_2^2\gamma_i}{\gamma_t\tau_t(\eta_i - L_{0} - L_{f})}\mathbb{E}[\| \delta_i^G\|_*^2].
\end{align*}
In view of \eqref{eq:1.52}, we have that the coefficient of the $\delta_t^G$ term on the right hand side of the above relation is strictly less than $1$. Moving the $\delta_t^G$ term to the left hand side and noting the conditions imposed on constants $R_1, R_2$, we have $$\mathbb{E}[\| \delta_t^G \|_*^2] \leqslant R_1 + R_2\textstyle\sum_{i=0}^{t-1}\mathbb{E}[\| \delta_i^G \|_*^2],$$ for all $t \leqslant T-1$. Using Lemma \ref{lemma9} for the above relation, we have \eqref{eq:1.53}. Hence we conclude the proof. 
\end{proof}
We are now ready to prove Theorem~\ref{thm11}. \textcolor{black}{Before we proceed, we remark that we the results from Lemma~\ref{lemma6} in Section~\ref{sec:aux} for the proof.}
\begin{proof}[\textbf{Proof of Theorem \ref{thm11}}] 
 It is easy to verify that $\{\gamma_t, \theta_t, \eta_t, \tau_t\}$ set according to Theorem~\ref{thm11} satisfies \eqref{eq:1.22}. Note that \eqref{eq:1.25} is satisfied if $4M_f^2 \leqslant \frac{\tau_t(\eta_{t-2} - L_{0} - L_{f})}{12}$. This follows due to the fact that $\{\eta_t\}$ is a non-decreasing sequence, $\theta_t = 1$ for all $t \geqslant 0$. Then we have 
 \begin{align*}
     \frac{\tau_t(\eta_{t-2} - L_{0} - L_{f})}{12} & \geqslant \frac{4 M_f}{D_X}12M_fD_X \times \frac{1}{12} = 4M_f^2
 \end{align*}
Also, since $(\eta_t - L_{0} - L_{f}) \geqslant \frac{24\|\sigma_{\nu}\|_2}{D_X}$ and $\tau_t \geqslant 8D_X\|\sigma_{\nu}\|_2$, we have
\begin{align*}
    \tau_t(\eta_t - L_{0} - L_{f}) \geqslant 192\|\sigma_{\nu}\|_2^2
\end{align*}
for all $t \geqslant 0$. In view of the above relation, we have 
\begin{align}
    \frac{96\|\sigma_{\nu}\|_2^2}{\tau_t(\eta_t - L_{0} - L_{f})} \leqslant \frac{1}{2}, \label{eq:1.57}
\end{align}
hence \eqref{eq:1.52} is satisfied. We also need to show the existence of $R_1$ and $R_2$ satisfying \eqref{eq:1.55} and \eqref{eq:1.56}, respectively. Using the fact that $\gamma_t, \eta_t$ and $\tau_t$ are constants for all $t \geqslant 0, \tau \eta \geqslant \frac{96T\sigma_{X, f}\|\sigma_{\nu}\|_2}{D_X}$ and noting \eqref{eq:1.57}, we obtain
\begin{align*}
    \left(1 - \frac{96\|\sigma_{\nu}\|_2^2}{\tau_t(\eta_t - L_{0} - L_{f})}\right)^{-1}\frac{96\|\sigma_{\nu}\|_2^2\gamma_i}{\gamma_t\tau_t(\eta_i - L_{0} - L_{f})} \leqslant 2 \frac{96\|\sigma_{\nu}\|_2^2}{\tau \eta} \leqslant 2\frac{\|\sigma_{\nu}\|_2D_X}{T\sigma_{X, f}} \leqslant \frac{2}{T},
\end{align*}
where in the last relation, we used the fact that $\sigma_{X, f} \geqslant D_X\|\sigma_{\nu}\|_2$. In view of the above relation and \eqref{eq:1.56}, we can set 
\begin{align}
    R_2 := \frac{2}{T}. \label{eq:1.58}
\end{align}
Noting \eqref{eq:1.55} along with the fact that $\mathcal{H}_* \geqslant \frac{L_{f}D_X[\|y^*\|_2 - 1]_+}{2}$, setting $y_0 = \mathbf{0}$, using \eqref{eq:1.57}, \eqref{eq:1.52}, $\gamma_t \tau_t = \tau \geqslant \sqrt{96T}\sigma_{X, f}, \sum_{i=0}^t\frac{\gamma_i}{\eta_i - L_{0} - L_{f}} = \frac{t+1}{\eta} \leqslant \frac{\sqrt{T}D_X}{\sqrt{2[\mathcal{H}_*^2 + \sigma_{0, \nu_0}^2 + 48\|\sigma_{\nu}\|_2^2]}}$, and $\sum_{i=1}^t\frac{\gamma_i\theta_i^2}{\tau_i} + \frac{\gamma_t}{\tau_t} = \frac{t+1}{\tau} \leqslant \frac{T}{\tau}$ for all $t \leqslant T-1$, we can see that the RHS of \eqref{eq:1.55} is at most 
\begin{equation}
\begin{aligned}[b]
   &  \hspace{-0.15in}2\biggl[2\sigma_{0, \nu_0}^2 + 48\|\sigma_{\nu}\|_2^2\biggl\{\frac{7}{12}\|y^*\|_2^2 + \frac{\eta}{\tau}D_X^2 + \frac{\sqrt{2T}D_X\mathcal{H}_*^2}{\sqrt{\mathcal{H}_*^2 + \sigma_{0, \nu_0}^2 + 48\|\sigma_{\nu}\|_2^2}}\frac{1}{\sqrt{96T}\sigma_{X, f}} + 12\sigma_{X, f}^2\frac{T}{\tau^2} \\ &\hspace{-0.2in} + \frac{[\nu_0^2L_0n + M_Xn(\textstyle\sum_{i=1}^m\nu_i^4L_i^2)^{1/2}]\sqrt{T}}{4\sqrt{6}\sigma_{X,f}}\biggr\}\biggr] \\
   & \hspace{-0.3in}\leqslant 2\left[2\sigma_{0, \nu_0}^2 + 48\|\sigma_{\nu}\|_2^2\left\{\frac{7}{12}\|y^*\|_2^2 + \frac{\eta}{\tau}D_X^2 + \frac{D_X\mathcal{H}_*}{\sqrt{48}\sigma_{X, f}} + 12T\sigma_{X, f}^2\frac{1}{96T\sigma_{X, f}^2} + \frac{[\nu_0^2L_0n + M_Xn(\textstyle\sum_{i=1}^m\nu_i^4L_i^2)^{1/2}]\sqrt{T}}{4\sqrt{6}\sigma_{X, f}}\right\}\right] \\
   &\hspace{-0.3in} \leqslant 2\biggl[2\sigma_{0, \nu_0}^2 + 48\|\sigma_{\nu}\|_2^2\biggl\{\frac{7}{12}\|y^*\|_2^2 + \frac{D_X}{\sigma_{X, f}}\left(\sqrt{\frac{[\mathcal{H}_*^2 + \sigma_{0, \nu_0}^2 + 48B^2\|\sigma_{\nu}\|_2^2]}{48}} + \frac{\mathcal{H}_*}{\sqrt{48}}\right)\\
   & \hspace{-0.3in}~~~~ + \frac{6\max\{2M_f, 4\|\sigma_{\nu}\|_2\}D_X}{2\max\{2M_f, 4\|\sigma_{\nu}\|_2\}}\frac{1}{D_X} + \frac{1}{8} + \frac{[\nu_0^2L_0n + M_Xn(\textstyle\sum_{i=1}^m\nu_i^4L_i^2)^{1/2}]\sqrt{T}}{4\sqrt{6}\sigma_{X, f}}\biggr\}\biggr] \\
   & \hspace{-0.3in}\leqslant 2\biggl[2\sigma_{0, \nu_0}^2 + 28\|\sigma_{\nu}\|_2^2\|y^*\|_2^2 + 150\|\sigma_{\nu}\|_2^2 + \sqrt{48}\|\sigma_{\nu}\|_2[2\mathcal{H}_* + (\sigma_{0, \nu_0} + \sqrt{48}\|\sigma_{\nu}\|_2)]  \\ & \hspace{-0.2in}+ 2\sqrt{6}D_X^{-1}\| \sigma_{\nu}\|_2  [\nu_0^2L_0n + M_Xn(\textstyle\sum_{i=1}^m\nu_i^4L_i^2)^{1/2}]\sqrt{T}\biggr] \\ 
   &\hspace{-0.3in} =: R_1
\end{aligned}
\end{equation}
where in the last inequality, we used the fact that $\frac{\|\sigma_{\nu}\|_2D_X}{\sigma_{X, f}} \leqslant 1$. Note that the last term in the above sequence of relations is a constant satisfying the requirement in \eqref{eq:1.55}. Hence, we can set 
\begin{align}
    R_1 := 2\biggl[2\sigma_{0, \nu}^2 + 28\|\sigma_{\nu}\|_2^2\|y^*\|_2^2 + 150\|\sigma_{\nu}\|_2^2 + \sqrt{48}\|\sigma_{\nu}\|_2[2\mathcal{H}_* + (\sigma_{0, \nu} + \sqrt{48}\|\sigma_{\nu}\|_2)] \notag \\ + 2\sqrt{6}D_X^{-1}\|\sigma_{\nu}\|_2[\nu_0^2L_0n + M_Xn(\textstyle\sum_{i=1}^m\nu_i^4L_i^2)^{1/2}]\sqrt{T}\biggr] \label{eq:1.59}
\end{align}
Then using Lemma \ref{lemma10} and noting \eqref{eq:1.58}, we have for all $t \leqslant T-1$
\begin{align*}
    \mathbb{E}[\|\delta_t^G\|_*^2] \leqslant \begin{cases} 4\sigma_{0, \nu_0}^2 & \text{ if } \|\sigma_{\nu}\|_2  = 0; \\ R_1\left(1 + \frac{2}{T}\right)^{T-1} \leqslant R_1e^2 & \text{ otherwise. } \end{cases}
\end{align*}
Noting the above relation, \eqref{eq:1.59} and the definition of $\zeta$, we have 
\begin{align}
    \mathbb{E}[\| \delta_t^G\|_*^2] \leqslant \zeta^2, \quad \forall t \leqslant T-1. \label{eq:1.60}
\end{align}
Hence, according to \eqref{eq:1.37} with $y_0 = \mathbf{0}$ and using \eqref{eq:1.60}, we have
\begin{align*}
    \mathbb{E}[f_0(\bar{x}_T) - f_0(x^*)] &\leqslant \frac{1}{T}\left[(\eta + L_{0} + L_{f})W(x^*, x_0) + \frac{2T\zeta^2}{\eta} + 12\sigma_{X, f}^2\frac{T}{\tau}\right] \\ &~~~+ [\nu_0^2L_0n + M_Xn(\textstyle\sum_{i=1}^m\nu_i^4L_i^2)^{1/2}].
\end{align*}
Using the bound $W(x^*, x_0) \leqslant D_X^2$, we obtain \eqref{eq:1.61}. From \eqref{eq:1.49} and \eqref{eq:1.60}, we have for $T \geqslant 1$
\begin{align*}
    \mathbb{E}\| [f(\bar{x}_T)]_+ \|_2 &\leqslant \frac{1}{T}\left[3(\|y^*\|_2 + 1)^2\tau + (\eta + L_{0} + L_{f})W(x^*, x_0) + \frac{2(\zeta^2 + \mathcal{H}_*^2)T}{\eta} + \frac{13\sigma_{X, f}^2T}{\tau} \right] \\
    &~~~~+ [\nu_0^2L_0n + M_Xn(\textstyle\sum_{i=1}^m\nu_i^4L_i^2)^{1/2}].
\end{align*}
Using bounds $W(x^*, x_0) \leqslant D_X^2$, we obtain \eqref{eq:1.62}. Define
\begin{align}
    \bar{\sigma_f}^2 & := 2(1+\sigma_f^2) \\
    \bar{\sigma}_0^2 & := 1 + 10(n+4)[\sigma_0^2 + [L_0(1+D_X) + M_{0}]^2] \\
    \bar{\sigma_i}^2 & := \frac{1}{m} + 10(n+4)[\sigma_i^2 + [L_i(1+D_X) + M_{i}]^2] \quad \text{ for } i \in \{1, \dots, m] \\
    \bar{\sigma}^2 & = 1 + 10(n+4)[ \|\sigma\|_2^2 + 2L_f^2(1+D_X)^2 + 2M_f^2] \\
    \overline{\sigma_{X, f}} & = (2(1+\sigma_f^2) + D_X^2\bar{\sigma}^2)^{1/2} \\
    \overline{\zeta} & := 2e\left\{\bar{\sigma_0}^2 + \overline{\sigma}^2(14\|y^*\|_2^2 + 75) + 2\sqrt{3}\overline{\sigma}(2\mathcal{H}_* + \bar{\sigma}_0 + \sqrt{48}\overline{\sigma}) + \sqrt{6}D_{X}^{-1}\overline{\sigma}\right\}^{1/2}.
\end{align}
By choice of $\nu_0, \nu_i$ for $i \in [m]$, definition of $\sigma_{f, \nu}^2$, $\tilde{B}_i$, $\sigma_{i, \nu_i}^2$, and $\sigma_{\nu}$, we have 
\begin{align*}
    \sigma_{f, \nu}^2 & \leqslant 2 + 2\sigma_f^2 =: \overline{\sigma}_{f}^2 \\
    \tilde{B}_i & \leqslant L_i(1 + D_X) + M_{i} \\
    \sigma_{0, \nu_0}^2 & \leqslant 1 + 10(n+4)[\sigma_0^2 + [L_0(1+D_X) + M_{f, 0}]^2] \\
    \sigma_{i, \nu_i}^2 & \leqslant \frac{1}{m} + 10(n+4)[\sigma_i^2 + [L_i(1+D_X) + M_{f,i}]^2] =: \overline{\sigma}_i^2 \quad \text{ for } i \in [m] \\
    \|\sigma_{\nu}\|_2^2 & \leqslant 1  + 10(n+4)[\| \sigma\|_2^2 + 2L_f^2(1+D_X)^2 + 2M_f^2] =: \overline{\sigma}^2.
\end{align*}
Furthermore, we also have that  $\nu_0^2L_0n + M_Xn\left(\textstyle\sum_{i=1}^m\nu_i^4L_i^2\right)^{1/2}  \leqslant \frac{1}{\sqrt{T}} $. Using these relations, we see that $\sigma_{X, f}  \leqslant \overline{\sigma_{X, f}}$ and $\zeta  \leqslant \overline{\zeta}$. Hence, we have 
\begin{align}
    \mathbb{E}[f_0(\bar{x_T}) - f_0(x^*)] & \leqslant \frac{(L_0+L_f)D_X^2 + \max\{12M_f, 24\overline{\sigma}\}D_X}{T} + \frac{1}{\sqrt{T}}\left[\sqrt{2(\mathcal{H}_*^2 + \bar{\sigma}_0^2 + 48\overline{\sigma}^2)}D_X + 1\right]\notag\\
    &\quad + \frac{1}{\sqrt{T}}\left\{\frac{\sqrt{2} D_X\overline{\zeta^2}}{\sqrt{\mathcal{H}_*^2 + \sigma_0^2 + 48\|\sigma\|_2^2}} + \frac{\sqrt{3}\overline{\sigma_{X,f}}}{\sqrt{2}}\right\} 
\end{align}
and
\begin{align}
  &  \mathbb{E}[\| [f(\bar{x}_T)]_+ \|_2] \leqslant  \frac{1}{\sqrt{T}} + \frac{(L_0 + L_f)D_X^2 + \max(12M_f, 24\overline{\sigma})D_X\left(1 + (\|y^*\|_2 + 1)^2\right)}{T}\notag \\
    &\quad \frac{1}{\sqrt{T}}\left\{\left[12\sqrt{6}(\|y^*\|_2 + 1)^2 + \frac{13}{4\sqrt{6}}\right]\overline{\sigma_{X, f}} + \sqrt{2}D_X\left[\sqrt{\mathcal{H}_*^2 + \overline{\sigma}_0^2 + 48\overline{\sigma}^2} + \frac{\overline{\zeta}^2 + \mathcal{H}_*^2}{\sqrt{\mathcal{H}_*^2 + \sigma_0^2 + 48\|\sigma\|_2^2}}\right]\right\} \notag
    \end{align}
As a consequence, to obtain an $(\varepsilon, \varepsilon)$-optimal solution with Algorithm $1$, we need the number of iterations to be
\begin{equation}
\begin{aligned}[b]
    T := \max\Biggl\{\frac{25}{\varepsilon^2}, \frac{5(L_0+L_f)D_X^2 + 5\max(12M_f, 24\overline{\sigma})D_X\left(1 +(\|y^*\|_2 + 1)^2\right)}{\varepsilon}, \\ \frac{\overline{\sigma_{X,f}^2}}{\varepsilon^2}\left[60\sqrt{6}(\|y^*\|_2 + 1)^2+ \frac{65}{4\sqrt{6}}\right]^2, \\\frac{50}{\varepsilon^2}\left[D_X\sqrt{\mathcal{H}_*^2 + \bar{\sigma}_0^2 + 48\overline{\sigma}^2} + \frac{D_X(\overline{\zeta^2} + \mathcal{H}_*^2)}{\sqrt{\mathcal{H}_*^2 + \sigma_0^2 + 48\|\sigma\|_2^2}}\right]^2\Biggr\}.
\end{aligned}
\end{equation}
Now, by the choice of $\nu_o$ and $\nu_i$ in~\eqref{nunot} and~\eqref{eq:nui} respectively, we see that the oracle complexity is given by $\mathcal{O}((m+1)n)/\epsilon^2).$
\end{proof}

\section{Auxiliary results}\label{sec:aux}
In this subsection, we state some auxiliary results from~\cite{boob2019proximal}, which we used in the proofs above. 
\begin{lemma}\cite[Lemma 2.4]{boob2019proximal}\label{lemma5}
Assume that $g: S \to \mathbb{R}$ satisfies
\begin{align}
    g(y) \geqslant g(x) + \langle g'(x), y - x\rangle + \mu W(y, x), \quad \forall x, y \in S \label{eq:1.6}
\end{align}
for some $\mu \geqslant 0$, where $S$ is convex set in $\mathbb{R}^n$. If $\bar{x} = \arg\min_{x \in S}\{g(x) + W(x, \tilde{x})\},$ then $g(\bar{x}) + W(\bar{x}, \tilde{x}) + (\mu + 1)W(x, \bar{x}) \leqslant g(x) + W(x, \tilde{x}), ~~\forall x \in S.$
\end{lemma}

\begin{lemma} \cite[Lemma 2.6]{boob2019proximal}\label{lemma7}
Let $\rho_0, \dots, \rho_j$ be a sequence of elements in $\mathbb{R}^n$ and let $S$ be a convex set in $\mathbb{R}^n$. Define the sequence $v_t, t = 0, 1, \dots$, as follows: $v_0 \in S$ and 
\begin{align*}
    v_{t+1} = \arg\min_{x \in S} \langle \rho_t, x\rangle + \frac{1}{2}\| x - v_t\|_2^2.
\end{align*}
Then for any $x \in S$ and $t \geqslant 0$, the following inequalities hold
\begin{align}
    \langle \rho_t, v_t - x\rangle &\leqslant \frac{1}{2}\| x - v_t\|_2^2 - \frac{1}{2}\| x - v_{t+1}\|_2^2 + \frac{1}{2}\|\rho_t\|_2^2, \label{eq:1.38} \\
    \textstyle\sum_{t=0}^j\langle \rho_t, v_t - x\rangle &\leqslant \frac{1}{2}\|x - v_0\|_2^2 + \frac{1}{2}\sum_{t=0}^j\|\rho_t\|_2^2. \label{eq:1.64}
\end{align}
\end{lemma}

\begin{lemma} \cite[Lemma 2.8]{boob2019proximal} \label{lemma9}
Let $\{a_t\}_{t \geqslant 0}$ be a nonnegative sequence, $m_1, m_2 \geqslant 0$ be constants such that $a_0 \leqslant m_1$ and the following relation holds for all $t \geqslant 1$: 
\begin{align*}
    a_t \leqslant m_1 + m_2\textstyle\sum_{k=0}^{t-1}a_k.
\end{align*}
Then we have $a_t \leqslant m_1(1 + m_2)^t$. 
\end{lemma}

The proof of the above three lemmas could be found in~\cite{boob2019proximal}. We also state and prove the following result, which is an adaptation of Lemma 2.5 in~\cite{boob2019proximal} to the zeroth-order setting. We highlight that the definition of the terms $q_t$, $\bar{q}_t$, $\delta_t^F$ and $\delta_t^G$ appearing in Lemma~\ref{lemma6} below are based on the stochastic zeroth-order gradient estimator (defined in~\eqref{eq:gradest}). Whereas, the corresponding terms from Lemma 2.5 in~\cite{boob2019proximal} are based on the stochastic first-order gradients (as \cite{boob2019proximal} deals with stochastic first-order optimization). This necessitates dealing with the Lipschitz continuity based arguments of the smoothed functions rather than the original functions as done in~\cite{boob2019proximal}. We do so by combining an argument from~\cite{nesterov2017random} on the analysis of stochastic zeroth-order method, along with the proof of Lemma 2.5 from \cite{boob2019proximal}. Hence, we provide a full proof of Lemma~\ref{lemma6} below for the convenience of readers who might be unfamiliar with the analysis of stochastic zeroth-order optimization algorithms. 

\begin{lemma} \cite[Lemma 2.5 adapted to the stochastic zeroth-order setting]{boob2019proximal} \label{lemma6}: ~Suppose Assumptions~\ref{Assumption1},~\ref{Assumption3} and ~\ref{Assumption2} are satisfied. Assume that $\{\gamma_t, \eta_t, \tau_t, \theta_t\}$ is a non-negative sequence satisfying 
\begin{align}\label{eq:1.22} 
    \gamma_t\theta_t  = \gamma_{t-1}, \qquad  \gamma_t\tau_t  \leqslant \gamma_{t-1}\tau_{t-1}, \qquad
    \tau_t\eta_t  \leqslant \gamma_{t-1}\eta_{t-1}, 
\end{align}
and 
\begin{align}
    (2M_{f})^2\frac{\theta_t}{\theta_{t-1}} & \leqslant \frac{\tau_t(\eta_{t-2} - L_{0} - L_{f})}{12}, \quad \theta_t(M_{f})^2 \leqslant \frac{\tau_t(\eta_{t-1} - L_{0} - L_{f})}{12}, \notag \\
    (2M_{f})^2\frac{1}{\theta_{T-1}} & \leqslant \frac{\tau_{T-1}(\eta_{T-2} - L_{0} - L_{f})}{12}, \quad M_{f}^2 \leqslant \frac{\tau_{T-1}(\eta_{T-1} - L_{0} - L_{f})}{12}, \label{eq:1.25}
\end{align}
where $M_f, L_f$ are defined in \eqref{eq:1.10}. Then, for all $T \geqslant 1$ and $z \in \{(x, y) : x \in X, y \geqslant \mathbf{0}\}$, we have
\begin{align}
    &~~\sum_{t=0}^{T-1}\gamma_tQ_{\nu}(z_{t+1}, z) + \sum_{t=0}^{T-1}\gamma_t[\langle \delta_t^G, x_t - x\rangle - \langle \delta_{t+1}^F, y_{t+1} - y\rangle] \notag \\
    \leqslant&~~\gamma_0 \eta_0 W(x, x_0) - \gamma_{T-1}\eta_{T-1}W(x, x_T) + \frac{\gamma_0\tau_0}{2}\| y - y_0 \|_2^2 - \frac{\gamma_{T-1}\tau_{T-1}}{12}\|y - y_T\|_2^2 \notag \\
    + &~~\textstyle\sum_{t=0}^{T-1}\frac{2\gamma_t}{\eta_t - L_{0} - L_{f}}\left[\| \delta_t^G\|_*^2 + \left(\frac{L_{f}D_X}{2}[\|y\|_2 - 1]_+\right)^2\right] \notag \\
     + &~~\textstyle\sum_{t=1}^{T-1}\frac{3\gamma_t\theta_t^2}{2\tau_t}\|q_t - \bar{q}_t\|_2^2 + \frac{3\gamma_{T-1}}{2\tau_{T-1}}\|q_T - \bar{q}_T\|_2^2. \label{eq:1.29}
\end{align}
Here $q_t := \ell_F(x_t) - \ell_F(x_{t-1}), \bar{q}_t := \ell_f(x_t) - \ell_f(x_{t-1})$, $\delta_t^F := \ell_F(x_t) - \ell_f(x_t)$ and $\delta_t^G := G_{0, \nu_0}(x_t, \xi_t, u_t) + \sum_{i \in [m]}G_{i, \nu_i}(x_t, \xi_t, u_t)y_{i, t+1} - f_{0, \nu_0}'(x_t) - \sum_{i=1}^mf_{i, \nu_i}'(x_t)y_{i, t+1}$, where $ G_{0, \nu_0}$ and  $G_{i, \nu_0}$, $i\in[m]$ are the stochastic zeroth-order gradients defined in~\eqref{eq:gradest}.
\end{lemma}
\begin{proof}[\textbf{Proof of Lemma~\ref{lemma6}}]
Note that $y_{t+1} = \arg\min_{y \geqslant \mathbf{0}}\langle -s_t, y\rangle + \frac{\tau_t}{2}\| y - y_t\|_2^2$. Hence, using Lemma \ref{lemma5} with $y \mapsto \langle -s_t, y\rangle$ and $\mu = 0$, we have for all $y \geqslant \mathbf{0}$,
\begin{align}
    -\langle s_t, y_{t+1} - y\rangle \leqslant \frac{\tau_t}{2}[\|y- y_t\|_2^2 - \|y_{t+1} - y_t\|_2^2 - \|y - y_{t+1}\|_2^2]. \label{eq:1.11}
\end{align}
Let us denote $v_t := f_{0, \nu_0}'(x_t) + \sum_{i \in [m]}f_{i, \nu_i}'(x_t)y_{i, t+1}$ and $V_t := G_{0, \nu_0}(x_t, \xi_t, u_t) + \sum_{i \in [m]}G_{i, \nu_i}(x_t, \xi_t, u_t)y_{i, t+1}$. Then using Lemma \ref{lemma5} with $x \mapsto \langle V_t, x\rangle$ and the optimality of $x_{t+1}$, we have for all $x \in X$,
\begin{align}
    \langle V_t, x_{t+1} - x\rangle \leqslant \eta_t[W(x, x_t) - W(x_{t+1}, x_t)] - \eta_tW(x, x_{t+1}). \label{eq:1.13}
\end{align}
Due to the convexity of $f_{0, \nu_0}$ and $f_{i, \nu_i}$, and since $f_0, f_i$ are Lipschitz, and by the definition of $\ell_f$, and the fact that $y_{t+1} \geqslant \mathbf{0}$, we have
\begin{align}
    & \langle v_t, x_{t+1} - x\rangle = \langle f_{0, \nu_0}'(x_t) + \textstyle\sum_{i \in [m]}f_{i, \nu_i}'(x_t)y_{i, t+1}, x_{t+1} - x\rangle \notag \\
    & = \langle f_{0, \nu_0}'(x_t), x_{t+1} - x_t + x_t - x\rangle + \langle f_{\nu}'(x_t)y_{t+1}, x_{t+1} - x_t + x_t - x\rangle \notag \\
    & \geqslant f_{0, \nu_0}(x_t) - f_{0, \nu_0}(x) + f_{0, \nu_0}(x_{t+1}) - f_{0, \nu_0}(x_t) - \frac{L_{0}}{2}\| x_{t+1} - x_t\|^2 \notag \\ & + \langle y_{t+1}, \ell_f(x_{t+1}) - f_{\nu}(x_t)\rangle + \langle y_{t+1}, f_{\nu}(x_t) - f_{\nu}(x)\rangle \notag \\
    & = f_{0, \nu_0}(x_{t+1}) - f_{0, \nu_0}(x) + \langle \ell_f(x_{t+1}) - f_{\nu}(x), y_{t+1}\rangle - \underbrace{\frac{L_{0}}{2}\|x_{t+1} - x_t\|^2}_{O_{t+1}}. \label{eq:1.12}
\end{align}
 Combining \eqref{eq:1.13}, \eqref{eq:1.12}, noting that $\delta_t^G = V_t - v_t$, we have 
\begin{align}
    &~~f_{0, \nu_0}(x_{t+1}) - f_{0, \nu_0}(x) + \langle \ell_f(x_{t+1}) - f_{\nu}(x), y_{t+1}\rangle + \langle \delta_t^G, x_{t+1} - x\rangle \notag \\ 
     \leqslant &~~\eta_tW(x, x_t) - \eta_tW(x_{t+1}, x_t) - \eta_tW(x, x_{t+1}) + O_{t+1}. \label{eq:1.15}
\end{align}
Noting the definition of $Q_{\nu}(\cdot, \cdot)$ (see \eqref{eq:1.14}) and, adding \eqref{eq:1.11} and \eqref{eq:1.15}, we obtain
\begin{align}
    &~~Q_{\nu}(z_{t+1}, z) - \langle f_{\nu}(x_{t+1}), y\rangle + \langle \ell_f(x_{t+1}), y_{t+1}\rangle - \langle s_t, y_{t+1} - y\rangle + \langle \delta_t^G, x_{t+1} - x\rangle \notag \\
    \leqslant&~~\frac{\tau_t}{2}[\| y - y_t\|_2^2 - \|y_{t+1} - y_t\|_2^2 - \| y - y_{t+1}\|_2^2]  \notag \\
    +&~~\eta_tW(x, x_t) - \eta_tW(x_{t+1}, x_t) - \eta_tW(x, x_{t+1}) + O_{t+1}. \label{eq:1.16}
\end{align}
Note that we also have $ f_{i, \nu_i}(x_{t+1}) - \ell_{f_i}(x_{t+1}) \leqslant \frac{L_{i}}{2}\|x_{t+1} - x_t\|^2$. Then, using Cauchy-Schwarz inequality and noting definitions of $L_{f}$, we have 
\begin{align*}
    \langle y, f_{\nu}(x_{t+1}) - \ell_f(x_{t+1})\rangle \leqslant \| y\|_2\underbracket{\frac{L_{f}}{2}\|x_{t+1} - x_t\|^2}_{C_{t+1}}.
\end{align*}
 Noting the above relation and definitions of $q_t$ and $\delta_{t+1}^F$, we have
\begin{align}
    & \langle \ell_f(x_{t+1}), y_{t+1} \rangle - \langle f_{\nu}(x_{t+1}), y\rangle - \langle s_t, y_{t+1} - y\rangle \notag \\
     \geqslant& \langle \ell_f(x_{t+1}), y_{t+1}\rangle - \langle \ell_f(x_{t+1}), y\rangle - \langle s_t, y_{t+1} - y\rangle - \|y\|_2C_{t+1} \notag \\
     = &\langle \ell_f(x_{t+1}) - s_t, y_{t+1} - y\rangle - \|y\|_2C_{t+1} \notag \\
     = & \langle \ell_f(x_{t+1}) - \ell_F(x_t) - \theta_t q_t, y_{t+1} - y\rangle - \| y\|_2C_{t+1} \notag \\
     = &\langle q_{t+1}, y_{t+1} - y\rangle - \theta_t \langle q_t, y_t - y \rangle -\theta_t\langle q_t, y_{t+1} - y_t\rangle - \langle \delta_{t+1}^F, y_{t+1} - y\rangle - \|y\|_2C_{t+1}. \label{eq:1.17}
\end{align}
Then
\begin{align}
    \|y\|_2C_{t+1} & = \frac{L_{f}}{2}(\|y\|_2 - 1)\| x_{t+1} - x_t\|^2 + \frac{L_{f}}{2}\| x_{t+1} - x_t\|^2 \notag  \\
    & \leqslant \frac{L_{f}}{2}[\|y\|_2 - 1]_+\| x_{t+1} - x_t\|^2 + \frac{L_{f}}{2}\|x_{t+1} - x_t\|^2 \notag \\
    & \leqslant \frac{L_{f}}{2}\|x_{t+1} - x_t\|^2 + \frac{L_{f}D_X}{2}[\|y\|_2 - 1]_+\|x_{t+1} - x_t\|. \label{eq:1.18}
\end{align}

By \eqref{eq:1.16}, \eqref{eq:1.17}, and \eqref{eq:1.18}, noting the definition of $O_{t+1}$ and using the relation $\frac{1}{2}\| a - b\|^2 \leqslant W(a, b)$, we have 
\begin{align}
    & Q_{\nu}(z_{t+1}, z) + \langle q_{t+1}, y_{t+1} - y\rangle - \theta_t\langle q_t, y_t - y\rangle + \langle \delta_t^G, x_t - x \rangle - \langle \delta_{t+1}^F, y_{t+1} - y\rangle \notag \\
     \leqslant &\theta_t\langle q_t, y_{t+1} - y_t \rangle - \langle \delta_t^G, x_{t+1} - x_t\rangle  \notag \\ 
     + & \eta_tW(x, x_t) - \eta_tW(x, x_{t+1}) + \frac{\tau_t}{2}[\| y - y_t\|_2^2 - \| y_{t+1} - y_t\|_2^2 - \| y - y_{t+1} \|_2^2] \notag \\
     - & (\eta_t - L_{0} - L_{f})W(x_{t+1}, x_t) + \frac{L_{f}D_X}{2}[\| y\|_2 - 1]_+\|x_{t+1} - x_t\|. \label{eq:1.19}
\end{align}
Multiplying \eqref{eq:1.19} by $\gamma_t$, summing them up from $t = 0$ to $T - 1$ with $T \geqslant 1$, we obtain
\begin{align}
    &~~\textstyle\sum_{t=0}^{T-1}\gamma_tQ_{\nu}(z_{t+1}, z) + \sum_{t=0}^{T-1}[\gamma_t\langle q_{t+1}, y_{t+1} - y\rangle - \gamma_t\theta_t \langle q_t, y_t - y\rangle] + \sum_{t=0}^{T-1}\gamma_t[\langle \delta_t^G, x_t - x\rangle - \langle \delta_{t+1}^F, y_{t+1} - y\rangle] \notag \\ 
     \leqslant &~~ \textstyle\sum_{t=0}^{T-1}[\gamma_t\theta_t\langle q_t - \bar{q}_t, y_{t+1} - y_t\rangle + \gamma_t\theta_t\langle \bar{q}_t, y_{t+1} - y_t\rangle + \langle \gamma_t\delta_t^G, x_t - x_{t+1}\rangle ] \notag \\
     +&~~ \textstyle\sum_{t=0}^{T-1}\left[\frac{\gamma_t\tau_t}{2}\| y - y_t\|_2^2 - \frac{\gamma_t\tau_t}{2}\| y - y_{t+1}\|_2^2\right] - \sum_{t=0}^{T-1}\frac{\gamma_t\tau_t}{2}\|y_{t+1} - y_t\|_2^2 \notag \\
    + &~~ \textstyle\sum_{t=0}^{T-1}[\gamma_t\eta_t W(x, x_t) - \gamma_t\eta_t W(x, x_{t+1})] \notag \\
     - &~~ \sum_{t=0}^{T-1}\left[\gamma_t(\eta_t - L_{0} - L_{f})W(x_{t+1}, x_t) - \gamma_t\underbracket{\left(\frac{L_{f}D_X}{2}[\|y\|_2 - 1]_+\right)}_{\mathcal{H}(y)}\|x_{t+1} - x_t\|\right], \label{eq:1.20}
\end{align}
where $\mathcal{H}(y) := \frac{L_{f}D_X}{2}[\|y\|_2 - 1]_+$. Now we focus our attention to handle the inner product terms of \eqref{eq:1.20}. Noting the definition of $\bar{q}_t$, we have 
\begin{align}
    \| \bar{q}_t \|_2 & = \| \ell_f(x_t) - \ell_f(x_{t-1}) \|_2 \notag \\
    & = \| f_{\nu}(x_{t-1}) + f_{\nu}'(x_{t-1})^T(x_t - x_{t-1}) - f_{\nu}(x_{t-2}) - f_{\nu}'(x_{t-2})^T(x_{t-1} - x_{t-2})\|_2 \notag \\ 
    & \leqslant \| f_{\nu}(x_{t-1}) - f_{\nu}(x_{t-2}) \|_2 + \| f_{\nu}'(x_{t-1})^T(x_t - x_{t-1})\|_2 + \| f_{\nu}'(x_{t-2})^T(x_{t-1} - x_{t-2})\|_2 \notag \\ & \leqslant 2M_f\| x_{t-1} - x_{t-2}\| + M_f\| x_t - x_{t-1}\|,
\end{align}
where we used the fact that $ \| f_{\nu}(x) - f_{\nu}(y)\|  \leqslant M_f\| x - y\|$ and $ \| [f_{\nu}'(x)]^T(y-x) \|_2  \leqslant M_f\|y-x\|$, which follows from the Assumptions \ref{Assumption3} and~\ref{Assumption2} and Theorem \ref{thm1}; see~\cite{nesterov2017random} for a similar argument.

Using the above relation for $\|\bar{q}_t\|_2$, we now obtain
\begin{align}
    & ~~\gamma_t\theta_t\langle \bar{q}_t, y_{t+1} - y_t \rangle - \frac{\gamma_t\tau_t}{3}\| y_{t+1} - y_t\|_2^2 - \frac{\gamma_{t-2}(\eta_{t-2} - L_{0} - L_{f})}{4}W(x_{t-1}, x_{t-2})  \\
  &   - \frac{\gamma_{t-1}(\eta_{t-1} - L_{0} - L_{f})}{4}W(x_t, x_{t-1}) \notag \\
    \leqslant & ~~\gamma_t\theta_t \|\bar{q}_t\|_2 \| y_{t+1} - y_t\|_2 - \frac{\gamma_t\tau_t}{3}\| y_{t+1} - y_t\|_2^2 \notag \\
     &~~-  \frac{\gamma_{t-2}(\eta_{t-2} - L_{0} - L_{f})}{4}W(x_{t-1}, x_{t-2}) - \frac{\gamma_{t-1}(\eta_{t-1} - L_{0} - L_{f})}{4}W(x_t, x_{t-1}) \notag \\
     \leqslant & 2M_{f}\gamma_t\theta_t \|x_{t-1} - x_{t-2}\| \|y_{t+1} - y_t\|_2 - \frac{\gamma_t\tau_t}{6}\|y_{t+1} - y_t\|_2^2 - \frac{\gamma_{t-2}(\eta_{t-2} - L_{0} - L_{f})}{4}W(x_{t-1}, x_{t-2}) \notag \\
      &~~ + M_{f}\gamma_t\theta_t\|x_t - x_{t-1}\| \|y_{t+1} - y_t\|_2 - \frac{\gamma_t\tau_t}{6}\|y_{t+1} - y_t\|_2^2 - \frac{\gamma_{t-1}(\eta_{t-1} - L_{0} - L_{f})}{4}W(x_t, x_{t-1}) \notag \\
    \leqslant & ~0, \label{eq:1.23}
\end{align}
where the last inequality follows by applying the relation $W(x, y) \geqslant \frac{1}{2}\| x - y\|$, Young's inequality $(2ab \leqslant a^2 + b^2)$ applied twice, once with 
\begin{align*}
    a & = \left(\frac{\gamma_t \tau_t}{6}\right)^{1/2}\|y_{t+1} - y_t\|, \quad b = \left(\frac{\gamma_{t-2}(\eta_{t-2} - L_{0} - L_{f})}{8}\right)^{1/2}\| x_{t-1} - x_{t-2}\| 
\end{align*}
and second time with 
\begin{align*}
    a & = \left(\frac{\gamma_t\tau_t}{6}\right)^{1/2}\| y_{t+1} - y_t \|, \quad b = \left(\frac{\gamma_{t-1}(\eta_{t-1} - L_{0} - L_{f})}{8}\right)^{1/2}\|x_t - x_{t-1}\|,
\end{align*}
and the fact that
\begin{align*}
    2M_{f}\gamma_t\theta_t & \leqslant \left\{\frac{\gamma_t\gamma_{t-2}\tau_t(\eta_{t-2} - L_{0} - L_{f})}{12}\right\}^{1/2} \quad \Leftrightarrow \quad (2M_{f})^2\frac{\theta_t}{\theta_{t-1}} \leqslant \frac{\tau_t(\eta_{t-2} - L_{0} - L_{f})}{12}, \\
    M_{f}^2\gamma_t^2\theta_t^2 & \leqslant \frac{\gamma_t\gamma_{t-1}\tau_t(\eta_{t-1} - L_{0} - L_{f})}{12} \quad \Leftrightarrow \quad M_{f}^2\theta_t \leqslant \frac{\tau_t(\eta_{t-1} - L_{0} - L_{f})}{12},
\end{align*}
where the equivalences follow due to \eqref{eq:1.22}. 
Using Young's inequality, Cauchy-Schwarz inequality and the relation $u^Tv \leqslant \| u \| \|v\|_*$, we have
\begin{align}
    \gamma_t\theta_t\langle q_t - \bar{q}_t, y_{t+1} - y_t\rangle - \frac{\gamma_t\tau_t}{6}\| y_{t+1} - y_t \|_2^2 & \leqslant \frac{3\gamma_t\theta_t^2}{2\tau_t}\| q_t - \bar{q}_t \|_2^2, \notag \\
    \langle \gamma_t\delta_t^G, x_t - x_{t+1}\rangle - \frac{\gamma_t(\eta_t - L_{0} - L_{f})}{4}W(x_{t+1}, x_t) & \leqslant \frac{2\gamma_t}{\eta_t - L_{0} - L_{f}}\| \delta_t^G\|_*^2, \label{eq:1.24} \\
    \gamma_t\mathcal{H}(y)\| x_{t+1} - x_t\| - \frac{\gamma_t(\eta_t - L_{0} - L_{f})}{4}W(x_{t+1}, x_t) & \leqslant \frac{2\gamma_t}{\eta_t - L_{0} - L_{f}}\mathcal{H}(y)^2. \notag
\end{align}
Using \eqref{eq:1.23} and \eqref{eq:1.24} for $t = 0, \dots, T-1$ inside \eqref{eq:1.20} and noting \eqref{eq:1.22}, we have 
\begin{align}
    & \sum_{t=0}^{T-1}\gamma_tQ_{\nu}(z_{t+1}, z) + \gamma_{T-1}\langle q_T, y_T - y\rangle + \sum_{t=0}^{T-1}\gamma_t[\langle \delta_t^G, x_t - x\rangle - \langle \delta_{t+1}^F, y_{t+1} - y\rangle] \notag \\
    \leqslant & \gamma_0\eta_0W(x, x_0) - \gamma_{T-1}\eta_{T-1}W(x, x_T) + \frac{\gamma_0\tau_0}{2}\|y - y_0\|_2^2 - \frac{\gamma_{T-1}\tau_{T-1}}{2}\| y - y_T\|_2^2 \notag \\
    & + \sum_{t=0}^{T-1}\left[\frac{3\gamma_t\theta_t^2}{2\tau_t}\|q_t - \bar{q}_t\|_2^2 + \frac{2\gamma_t}{\eta_t - L_{0} - L_{f}}\|\delta_t^G\|_*^2 + \frac{2\gamma_t}{\eta_t - L_{0} - L_{f}}\mathcal{H}(y)^2\right] \notag \\
    & - \frac{\gamma_{T-2}(\eta_{T-2} - L_{0} - L_{f})}{4}W(x_{T-1}, x_{T-2}) - \frac{\gamma_{T-1}(\eta_{T-1} - L_{0} - L_{f})}{2}W(x_T, x_{T-1}), \label{eq:1.28}
\end{align}
where in the left hand side of the above relation, we used the fact that $q_0 = \ell_F(x_0) - \ell_F(x_{-1}) = \mathbf{0}$. Similarly, we see that $\bar{q}_0 = \mathbf{0}$. Hence, we can ignore $\| q_0 - \bar{q}_0 \|_2^2$ term in the right hand side of the above relation, after which we obtain
\begin{align}
    & -\gamma_{T-1}\langle \bar{q}_T, y_T - y \rangle - \frac{\gamma_{T-1}\tau_{T-1}}{3}\| y - y_T\|_2^2 \notag  \\
    & - \frac{\gamma_{T-2}(\eta_{T-2} - L_{0} - L_{f})}{4}W(x_{T-1}, x_{T-2}) - \frac{\gamma_{T-1}(\eta_{T-1} - L_{0} - L_{f})}{2}W(x_T, x_{T-1}) \notag \\
     \leqslant &~~ M_{f}\gamma_{T-1}\| x_T - x_{T-1}\| \|y_T - y\|_2 - \frac{\gamma_{T-1}\tau_{T-1}}{12}\| y - y_T\|_2^2 - \frac{\gamma_{T-1}(\eta_{T-1} - L_{0} - L_{f})}{2}W(x_T, x_{T-1}) \notag \\ & + 2M_{f}\gamma_{T-1}\| x_{T-1} - x_{T-2}\| \| y_T - y\|_2 - \frac{\gamma_{T-1}\tau_{T-1}}{6}\| y - y_T \|_2^2 - \frac{\gamma_{T-2}(\eta_{T-2} - L_{0} - L_{f})}{4}W(x_{T-1}, x_{T-2}) \notag \\
    & - \frac{\gamma_{T-1}\tau_{T-1}}{12}\| y_T - y\|_2^2 \notag \\
     \leqslant & ~~-\frac{\gamma_{T-1}\tau_{T-1}}{12}\|y_T - y\|_2^2, \label{eq:1.26}
\end{align}
where the last relation follows from \eqref{eq:1.25}, Young's inequality and the fact that 
\begin{align*}
    2M_{f}\gamma_{T-1} & \leqslant \left\{\frac{\gamma_{T-2}\gamma_{T-1}\tau_{T-1}(\eta_{T-2}- L_{0} - L_{f})}{12}\right\}^{1/2} \quad \Leftrightarrow \quad (2M_{f})^2\frac{1}{\theta_{T-1}} \leqslant \frac{\tau_{T-1}(\eta_{T-2} - L_{0} - L_{f})}{12} \\
    M_{f}\gamma_{T-1} & \leqslant \left\{\frac{\gamma_{T-1}^2\tau_{T-1}(\eta_{T-1} - L_{0} - L_{f})}{12}\right\}^{1/2} \quad \Leftrightarrow \quad M_{f}^2 \leqslant \frac{\tau_{T-1}(\eta_{T-1} - L_{0} - L_{f})}{12}.
\end{align*}
Moreover, again using Young's inequality and Cauchy-Schwarz inequality, we have 
\begin{align}
    -\gamma_{T-1}\langle q_T - \bar{q}_T, y_T - y\rangle - \frac{\gamma_{T-1}\tau_{T-1}}{6}\| y - y_T\|_2^2 \leqslant \frac{3\gamma_{T-1}}{2\tau_{T-1}}\| q_T - \bar{q}_T\|_2^2. \label{eq:1.27}
\end{align}
Using \eqref{eq:1.26} and \eqref{eq:1.27} in relation \eqref{eq:1.28}, noting that $q_0 - \bar{q}_0 = \mathbf{0}$ and replacing the definition of $\mathcal{H}(y)$, we obtain \eqref{eq:1.29}, which completes the proof.
\end{proof}

\end{document}